\documentclass{amsart}

\usepackage{mathrsfs}
\usepackage[utf8]{inputenc}
\usepackage{amsfonts}
\usepackage{amsmath}
\usepackage{amsthm}
\usepackage{amssymb}
\usepackage{latexsym}
\usepackage{enumerate}
\usepackage{enumitem}
\usepackage{color}

\newtheorem{theorem}{Theorem}
\newtheorem{lemma}[theorem]{Lemma}
\newtheorem{proposition}[theorem]{Proposition}
\newtheorem{corollary}[theorem]{Corollary}

\theoremstyle{definition}
\newtheorem{remark}[theorem]{Remark}

\newtheorem{question}[theorem]{Question}

\newtheorem{definition}[theorem]{Definition}

\newtheoremstyle{TheoremNum}
  {\topsep}{\topsep}              
  {\itshape}                      
  {}                              
  {\bfseries}                     
  {.}                             
  { }                             
  {\thmname{#1}\thmnote{ \bfseries #3}}
\theoremstyle{TheoremNum}
\newtheorem{reptheorem}{Theorem}

\newtheoremstyle{TheoremNum}
  {\topsep}{\topsep}              
  {\itshape}                      
  {}                              
  {\bfseries}                     
  {.}                             
  { }                             
  {\thmname{#1}\thmnote{ \bfseries #3}}
\theoremstyle{TheoremNum}
\newtheorem{replemma}{Lemma}

\begin{document}

\newcommand{\cc}{\mathfrak{c}}
\newcommand{\N}{\mathbb{N}}
\newcommand{\M}{\mathbb{M}}
\newcommand{\C}{\mathbb{C}}
\newcommand{\Q}{\mathbb{Q}}
\newcommand{\R}{\mathbb{R}}
\newcommand{\T}{\mathbb{T}}
\newcommand{\st}{*}
\newcommand{\PP}{\mathbb{P}}
\newcommand{\forces}{\Vdash}
\newcommand{\dom}{\text{dom}}
\newcommand{\osc}{\text{osc}}
\newcommand{\F}{\mathcal{F}}
\newcommand{\A}{\mathcal{A}}
\newcommand{\B}{\mathcal{B}}
\newcommand{\BB}{\mathscr{B}}
\newcommand{\MM}{\mathscr{M}}
\newcommand{\SSS}{\mathscr{S}}
\newcommand{\I}{\mathcal{I}}
\newcommand{\CC}{\mathcal{C}}
\newcommand{\X}{\mathcal{X}}
\newcommand{\K}{\mathbb{K}}
\newcommand{\seq}{2^{<{\mathbb N}}}
\newcommand{\can}{2^{{\mathbb N}}}
\newcommand{\sub}{{[\mathbb N]^\omega}}
\newcommand{\fin}{{[\mathbb N]^{<\omega}}}
\newcommand{\Id}{\operatorname{Id}}

\newcommand{\marg}[1]{\marginpar{\tiny\color{blue} #1}} 

\author[P.~Koszmider]{Piotr Koszmider}
\address[P.~Koszmider]{Institute of Mathematics, Polish Academy of Sciences,
ul. \'Snia\-deckich~8,  00-656~Warszawa, Poland}
\email{\texttt{piotr.koszmider@impan.pl}}
\thanks{The main part  of the research presented in this paper was carried out 
while the first named author visited Lancaster University in the early months of 2019.
He wishes to thank the second named author for hosting him and 
the Faculty of Science and Technology of Lancaster University for supporting the visit through the
Distinguished Visitor Programme Fund}

\author[N.~J.~Laustsen]{Niels Jakob Laustsen}
\address[N.~J.~Laustsen]{Department of Mathematics and Statistics, Fylde College,
 Lancaster University, Lancaster, LA1~4YF, United Kingdom}
\email{\texttt{n.laustsen@lancaster.ac.uk}}

\subjclass[2010]{%
03E05,
46E15,
47L10,
54D80
(primary); 
46B26, 
46H40,
47B38,
47L20,
54D45}
%
\keywords{Banach space of continuous functions, compact Hausdorff space,  $\Psi$\nobreakdash-space, almost disjoint family,
partitioner, perfect set property, bounded linear
  operator, closed operator ideal.}

\title[Banach spaces induced by almost disjoint families]{A Banach space
 induced by an almost disjoint family, admitting only few operators and decompositions}

\begin{abstract} We consider  the closed subspace of~$\ell_\infty$ generated by~$c_0$ and the characteristic functions of elements
  of an uncountable,
  almost disjoint family~$\A$ of infinite subsets of~$\N$. This Banach space has the form $C_0(K_\A)$ for
a locally compact Hausdorff space~$K_\A$ that is known under many
names, including $\Psi$\nobreakdash-space and Isbell--Mr\'ow\-ka space.

We construct an uncountable, almost disjoint family $\A$ such that the
algebra of all bounded linear operators on $C_0(K_\A)$ is as small as
possible in the precise sense that every bounded linear operator
on~$C_0(K_\A)$ is the sum of a scalar multiple of the identity and an
operator that factors through~$c_0$ (which in this case is equivalent
to having separable range). This implies that~$C_0(K_\A)$ has the
fewest possible decompositions: whenever~$C_0(K_\A)$ is written as the
direct sum of two infinite-dimensional Banach spaces~$\mathcal{X}$
and~$\mathcal{Y}$, either $\mathcal{X}$ is isomorphic to~$C_0(K_\A)$
and $\mathcal{Y}$ to~$c_0$, or vice versa. These results improve
previous work of the first named author in which an extra
set-theoretic hypothesis was required. We also discuss the
consequences of these results for the algebra of all bounded linear
operators on our Banach space $C_0(K_\A)$ concerning the lattice of
closed ideals, characters and automatic continuity of
ho\-mo\-mor\-phisms.

To exploit the perfect set property for Borel sets as in the classical
construction of an almost disjoint family by Mr\'owka, we need to deal
with $\N \times \N$ matrices rather than with the usual partitioners
of an almost disjoint family.  This noncommutative setting requires
new ideas inspired by the theory of compact and weakly compact
operators and the use of an extraction principle due to van Engelen,
Kunen and Miller concerning Borel subsets of the square.\smallskip

Accepted for publication in \emph{Advances in Mathematics.}
\end{abstract}

\maketitle

\section{Introduction}
\noindent
The symbols $\sub$  and $\fin$ denote the families  of all infinite and finite 
subsets, re\-spec\-tive\-ly, of the set $\N=\{0,1,2\,\ldots\}$ of nonnegative integers. A family $\A\subseteq \sub$  is called \emph{almost
  disjoint} if  $A\cap A'\in \fin$ whenever $A,A'\in\A$ are distinct.

We shall investigate the impact that the combinatorial structure of an almost disjoint family $\A$ has 
on the Banach space $\mathcal{X}_\A$
 associated
with it. The formal definition of this space is as follows.

\begin{definition}\label{definition-psi}
 Let $\A\subseteq \sub$ be an  almost disjoint family.
Then~$\X_\A$ is the closed subspace of~$\ell_\infty$ spanned by 
$\{1_A: A\in \A\cup\fin\}.$
\end{definition}
Here, and elsewhere, $1_A$ stands for the characteristic function of a
set $A$.  Banach spaces of the form~$\mathcal{X}_\A$ were first
considered by Johnson and Lindenstrauss in Example~2
of~\cite{johnson-lindenstrauss}. 
Being a
self-adjoint subalgebra of~$\ell_\infty$, $\mathcal{X}_\A$~is isometrically
isomorphic to~$C_0(K_\A)$, where~$K_\A$ is the Gelfand space
of~$\mathcal{X}_\A$. Moreover, $C_0(K_\A)$  contains a complemented copy
of~$c_0$, so $C_0(K_\A)$ is isomorphic to
the Banach space~$C(\alpha K_\A)$ of all continuous scalar-valued
functions on the one-point compactification~$\alpha K_\A$ of~$K_\A$.

The study of~$K_\A$ as an interesting example of a scattered, locally
compact Hausdorff space, which is nonmetrizable whenever~$\A$ is
uncountable, can be traced back to Alexandroff and
Urysohn~\cite{alexandroff}. Spaces of the form~$K_\A$ are known under
many names, including \emph{$\Psi$-spaces, Isbell--Mr\'owka spaces,
  AU-compacta} and \emph{Mr\'owka spaces}. We refer to~\cite{hrusak1}
and~\cite{hrusak2} for recent surveys on these spaces and their
numerous applications found during many decades of investigations.

It is often assumed in the literature that the almost disjoint
family~$\A$ inducing the space~$K_\A$ is maximal with respect to
inclusion, which corresponds to $K_\A$ being pseudo\-compact. We do
not make this assumption here, and indeed the families that we obtain are not maximal.

In \cite{mrowka}, the first named author used the continuum hypothesis, or its weakening 
$\mathfrak{p}=2^\omega$, to construct an uncountable,  almost disjoint family 
$\A\subseteq\sub$ such that every bounded linear operator on~$C_0(K_\A)$ has
the form $\lambda \operatorname{Id}+S$, where $\lambda$ is a scalar,  $\operatorname{Id}$~de\-notes the identity operator on~$C_0(K_\A)$, and~$S$ is an operator that factors through~$c_0$ in the sense that $S = VU$ for some  bounded linear operators $U\colon C_0(K_\A)\to c_0$ and $V\colon c_0\to C_0(K_\A)$.
Here we provide another construction of such a family~$\A$ that does not
require any additional set-theoretic axioms.

\begin{theorem}\label{main-theorem} There is an uncountable, almost disjoint family
$\A\subseteq\sub$ such that every bounded linear operator $T:C_0(K_\A)\rightarrow C_0(K_\A)$ has the form
  \[ T=\lambda \operatorname{Id}+S \]
for some scalar~$\lambda$ and some operator $S:C_0(K_\A)\rightarrow C_0(K_\A)$ that factors through~$c_0$.  
\end{theorem}

Note in particular that the operator~$S$ has separable range. In fact, for operators on a space of the form~$C_0(K_\A)$, it is well known that having separable range and factoring through~$c_0$ are equivalent properties; see Lemma~\ref{jl-c0} for details. 

Theorem~\ref{main-theorem} answers a question raised in~\cite{mrowka}, and together
with  results of 
Argyros and Raikoftsalis~\cite{argyros},
Mar\-ciszew\-ski and Pol~\cite{mpol}, and
Cabello S\'anchez, Castillo, Marciszewski, 
Plebanek and Salguero-Alarc\'on~\cite{sailing}, it  completes the list of solutions to the problems left open in~\cite{mrowka}.

Theorem~\ref{main-theorem} has some remarkable consequences, notably: 
\begin{enumerate}[label={\normalfont{(\roman*)}}]
\item\label{TrivDecomp} Whenever the Banach space~$C_0(K_\A)$ is
  decomposed into a direct sum $C_0(K_\A)=\mathcal X\oplus\mathcal Y$ of two closed, infinite-dimensional
  subspaces~$\mathcal{X}$ and~$\mathcal{Y}$, either~$\mathcal{X}$ is
  isomorphic to~$C_0(K_\A)$ and~$\mathcal{Y}$ is isomorphic to~$c_0$,
  or vice versa  (see
  Lemma~4 of~\cite{mrowka}).
\item\label{KKclosedideals}   The Banach algebra~$\BB(C_0(K_\A))$ of all bounded linear
operators on the Ba\-nach space~$C_0(K_\A)$
contains precisely four closed two-sided ideals, name\-ly
\[ \{0\}\subsetneq \mathscr{K}(C_0(K_\A))\subsetneq \mathscr{X}(C_0(K_\A))\subsetneq\BB(C_0(K_\A)), \]
where $\mathscr{K}(C_0(K_\A))$ and $\mathscr{X}(C_0(K_\A))$ are the ideals of compact operators and operators with separable range, respectively. This result is due to Kania and Kochanek (see Theorem~5.5 of~\cite{tomek-tomek}) and  was also observed by Brooker (unpublished).
\item\label{ker-character} The Banach algebra~$\BB(C_0(K_\A))$
admits a character, that is, a nonzero multiplicative linear functional, whose kernel is
$\mathscr{X}(C_0(K_\A))$; we shall discuss this point in much more detail in Subsection~\ref{few}.
\item\label{cont-maps} If $\phi: \alpha K_\A\rightarrow \alpha K_\A$ is a continuous map,
then either $\phi$ has countable range or 
all but countably many points of $\alpha K_\A$ are fixed by~$\phi$ (see Proposition~\ref{maps}). 
\end{enumerate}

Roughly speaking, these results say that~$C_0(K_\A)$ is ``minimal'' in
terms of its de\-compo\-si\-tions, closed operator ideals,  kernels of characters on
$\BB(C_0(K_\A))$ and continuous maps on $\alpha K_\A$. 
 Indeed, $C_0(K)$ contains a complemented copy of~$c_0$
when\-ever~$K$ is a scattered, locally compact Hausdorff space, so decompositions of~$C_0(K_\A)$ of the
form described in~\ref{TrivDecomp} are unavoidable, as is the 
ideal~$\mathscr{X}(C_0(K_\A))$ of operators factoring through~$c_0$ lying strictly
between~$\mathscr{K}(C_0(K_\A))$ and~$\mathscr{B}(C_0(K_\A))$
in~\ref{KKclosedideals}; and Proposition \ref{kernelofcharBCK} will show
that the kernel of the character on~$\mathscr{B}(C_0(K_\A))$ described in~\ref{ker-character} is as small as it can be.

We note that the conclusion of
Theorem~\ref{main-theorem} and its consequences~\ref{TrivDecomp}, \ref{KKclosedideals} and~\ref{ker-character} also hold true if~$C_0(K_\A)$ is replaced
with~$C(\alpha K_\A)$ because the two spaces are isomorphic.

Our construction of the almost disjoint family~$\A$ will exploit a well-known approach due to
Mr\'owka~\cite{mrowkas} that  starts from a
Borel almost disjoint family $\A\subseteq [\N]^\omega$ and then takes advantage of the perfect
set property for Borel sets. However, in the noncommutative setting
of operators rather than  continuous functions, instead of the usual partitioners
of an almost disjoint family, we need to consider $\N\times \N$ matrices representing operators.
The behaviour of such a matrix on the set  $A\cup B$ for $A, B\in \A$ is not completely determined by its behaviour on~$A$ and~$B$ alone, as it also
depends on the  entries of the matrix  belonging to $(A\setminus B)\times (B\setminus A)$ and $(B\setminus A)\times (A\setminus B)$. Overcoming these obstacles requires us to find a new
approach, which is inspired by the theory of compact and weakly compact operators on~$c_0$ and~$\ell_\infty$, and also to work with Borel subsets of Cartesian products.

\subsection*{Organization.} 
 The paper is organized as follows: we begin by recalling the main
 topological properties of spaces of the form~$K_\A$ in
 Section~\ref{section2}, before we sketch the key parts of Mr\'owka's
 original argument that our construction is modelled after. In the
 remainder of Section~\ref{section2} we introduce some basic notions
 and establish Theorem~\ref{main-theorem} subject to a sequence of 
 lemmas, whose proofs are then given in 
 Sections~\ref{scattered}, \ref{matrices-section} and~\ref{borel-section}.

  Theorem~\ref{main-theorem} and its consequences interact with a considerable body of known results 
concerning set theory, logic, topology, Banach spaces and Banach algebras. 
We discuss this in detail in 
Section \ref{final}, including the structure of continuous self-maps of~$\alpha K_\A$ and 
 the minimality 
 of the kernel of the character on~$\BB(C_0(K_\A))$  mentioned in items~\ref{cont-maps} and~\ref{ker-character} above, as well as
 the automatic continuity of homomorphisms from $\BB(C_0(K_\A))$ into a Banach algebra and 
 the place of our
 Banach space
 $C_0(K_\A)$ with few operators among other well-known Banach spaces with few operators.  Section~\ref{final} also contains some
 open questions which arise naturally from our results, including a discussion of their context and motivation.

\subsection*{Notation and terminology.} Our notation and terminology 
is mostly standard. We list the key parts here.
None of our arguments concerning Banach spaces depend on whether the space is over the real or complex
numbers, so the scalar field will be denoted generically by~$\K$.
The Banach algebra of all bounded linear operators
on a Banach space~$\mathcal{X}$ is denoted~$\BB(\mathcal{X})$, and the
identity operator on~$\mathcal{X}$ is~$\Id_{\mathcal{X}}$, or simply~$\operatorname{Id}$ if~$\mathcal X$ is clear from the context. We write
$\mathcal{X}\sim\mathcal{Y}$ to signify that the Banach
spaces~$\mathcal{X}$ and~$\mathcal{Y}$ are isomorphic. 

All topological spaces that we consider are 
Hausdorff. 
For a locally compact Hausdorff space~$K$, $C_0(K)$
denotes the Banach
space of all  $\K$-valued continuous functions~$f$ on $K$ such that the set
$\{x\in K: \lvert f(x)\rvert\geq \varepsilon\}$ is compact for each $\varepsilon>0$, endowed with the supremum norm.
By $1_A$ we mean the characteristic function of a set $A$. 
This notation is slightly  ambiguous as the domain of the function~$1_A$ depends on the ambient set
of $A$; it should always be clear from the context. We write $f|X$
for the restriction of a function $f$ to a sub\-set~$X$ of its domain.

The  \emph{Cantor cube} is denoted by~$\can$. It is the set of all
sequences of zeros and ones, and  is a compact Hausdorff space with respect to the product topology. For a subset~$A$ of~$\N$, we may naturally regard $1_A$ as an element of~$\can$.  

Given  a scalar sequence $(a_n)_{n\in \N}$ 
and a set $A\in \sub$, we write $\lim_{n\in A}a_n=b$
to mean that the sequence $(a_{n_k})_{k\in\N}$ converges and has limit~$b$, where 
$(n_k)_{k\in \N}$ is the increasing enumeration of~$A$.

The symbol $\mathfrak c$ will stand for the cardinality of the
continuum.
\section{Outline of the proof of Theorem~\ref{main-theorem}}\label{section2}
\noindent
In this section we shall prove Theorem~\ref{main-theorem} subject to a sequence of 
lemmas whose proofs will then be given in the subsequent sections  of the paper. The purpose of this organization is to give a clear overview of the proof before we go into the details.
We work with the Banach space~$\X_\A$ using its representation as a space of the form~$C_0(K_\A)$ for a locally compact Hausdorff space~$K_\A$,  so we begin by recalling the details of this representation.

\subsection{The locally compact space $K_\A$ induced by an almost disjoint family~$\A$}\label{psi-section}
 Recall that the Banach space $\mathcal{X}_\A$ was defined in Definition~\ref{definition-psi}.

\begin{definition} Let $\A\subseteq\sub$ be an almost disjoint family. Then
$K_\A$ denotes the topological space consisting of distinct
points $\{x_n: n\in \N\}\cup\{y_A: A\in \A\}$, where~$x_n$ is isolated
for every $n\in \N$ and the sets
$$U(A, F)=\{x_n: n\in A\setminus F\}\cup\{y_A\}$$
for $F\in\fin$ form a neighbourhood basis at each point $y_A$ for 
$A\in \A$. We write $U(A)$ for $U(A, \emptyset)$.
\end{definition}

The following lemma summarizes well known consequences of standard general
topological results (see \cite{hrusak1, hrusak2}, and \cite{engelking} for general background).

\begin{lemma}\label{topological} Let $\A\subseteq\sub$ be an almost disjoint family. Then:
  \begin{itemize}
  \item $K_\A$ is a locally compact, scattered Hausdorff space.
  \item   $K_\A$ is compact  if
    and only if $\A$ and $\N\setminus\bigcup\A$ are both finite.
  \item $\{x_n: n\in \N\}$ is the set of isolated points of~$K_\A;$ it is dense in $K_\A$, and so $K_\A$ is separable.
        Hence $K_\A$ is metrizable if and only it is second countable, if and only if $\A$ is countable. 
\item The subspace
$K_\A\setminus \{x_n: n\in \N\} = \{y_A: A\in \A\}$ 
  is closed and discrete.
  \item The sequence $(x_n)_{n\in A}$ converges to $y_A$ in $K_\A$ for every
    $A\in \A$.
\end{itemize}    
\end{lemma}

\begin{lemma} The Banach space $\X_\A$ is isometric to the space 
of all continuous functions on $K_\A$ vanishing at infinity.
\end{lemma}
\begin{proof} Consider the operator
$T: C_0(K_\A)\rightarrow \ell_\infty$ defined by $T(f)(n)=f(x_n)$ for each $n\in \N$.
We have $T(1_{U(A, F)})=1_{A\setminus F}\in \X_\A$ and $T(1_{\{ x_n\}})=1_{\{n\}}\in \X_\A$. 
All compact open sets of $K_\A$ are finite unions of sets of the form
$U(A, F)$ and finite subsets of $\{x_n: n\in \N\}$, so the range 
of $T$ is $\X_\A$.
By the density of $\{x_n: n\in \N\}$ in~$K_\A$, the operator $T$ is an isometry.
\end{proof}

\subsection{Constructing 
$C_0(K_\A)$ with few multiplication operators}\label{few-multiplications-section}

This subsection is not needed for obtaining the proofs of any results
in the remainder of this paper.  That is why the proofs are omitted or
sketched.  We provide this subsection as a motivation for the next
one. It is an extraction of some ideas from \cite{mrowka}
(cf.\ \cite{hrusak1}), which are later modified.  The main result of
this subsection is Theorem \ref{existence-F} where the existence of an
almost disjoint family $\B\subseteq\sub$ is proved such that~$K_\B$
has few bounded continuous scalar-valued functions, so in particular
$C_0(K_\B)$ has few multiplication operators. In the next subsection
we show how to modify these ideas to obtain $C_0(K_\B)$ with few
operators.

The main idea is to consider all potential
 bounded continuous functions~$F$ on~$K_\B$ by representing 
them as elements~$f$ of~$\ell_\infty$ given by $f(n)=F(x_n)$.  We may enumerate these elements
as $\{f_\xi: \xi<\mathfrak c\}$.
The family
 $\B=\{B_\xi: \xi<\mathfrak c\}$ is constructed by a transfinite induction of
length $\mathfrak c$, making sure that $B_\xi$ rejects $f_\xi$, i.e., it witnesses
that $f_\xi$ does not represent a bounded continuous function on $K_\B$ unless
$f_\xi$ is of the form mentioned in Theorem \ref{existence-F}.
This can be achieved, for example, by choosing $B_\xi$ such that $\lim_{k\in B_\xi}f_\xi(k)$
does not exist, because $\lim_{k\in B_\xi}F(x_k)$ must exist and be
equal to $F(y_{B_\xi})$
for every continuous function $F\colon K_\B\to\K$  since $(x_k)_{k\in B_\xi}$ converges
 to $y_{B_\xi}$ in $K_\B$ by Lemma \ref{topological}. 

 It is useful to introduce the following terminology:
\begin{definition}\label{admits-F} Suppose that $\A\subseteq\sub$ and $A\in\sub$. We say that:
\begin{itemize}
\item $A$ \emph{admits} $f\in \ell_\infty$ if $\lim_{n\in A}f(n)=0$.
\item $\A$ \emph{admits} $f$ if $A$ admits $f$ for every $A\in \A$.
\item $A$ \emph{rejects} $f\in \ell_\infty$ if  $\lim_{n\in A}f(n)$ does not exist.
\item $\A$ \emph{rejects} $f$ if $A$ rejects $f$ for some $A\in \A$.
\end{itemize}
\end{definition}

\begin{lemma}\label{monotonicity} Suppose that $f\in \ell_\infty$ and that
 $A, A'\in\sub$.
\begin{enumerate}
\item \emph{(Monotonicity)} If $\{A, A'\}$ admits $f$, then
$A\cup A'$ admits $f$.
\item \emph{(Decidability)} If $A$ does not  admit $f-\lambda 1_\N$ for any $\lambda\in \K$, 
then  $A'$ rejects~$f$ for every $A'\supseteq A$.
\end{enumerate}
\end{lemma}

The following lemma explains what properties
of an almost disjoint family in terms of rejection and admission
are needed to obtain Theorem \ref{existence-F}.

\begin{lemma}[Reduction]\label{reduction}
Suppose that  $\B\subseteq\sub$ is an almost disjoint family such that, for every $f\in \ell_\infty$,
either $\B$ rejects $f$ or there is $\lambda_f\in \K$ and 
a countable subset $\B_f\subseteq \B$ such that $\B\setminus \B_f$ admits $f-\lambda_f 1_\N$.
Then all bounded continuous functions $F:K_\B\rightarrow \K$ are of the form
$$F=\lambda 1_{K_\B}+ G,$$
 where $\lambda\in \K$ and $G$ is nonzero only on countably
many points of $K_\B$.
\end{lemma}
\begin{proof}
Let $F: K_\B\rightarrow\K$ be a bounded continuous function, and
let $f\in \ell_\infty$ be given by $f(n)=F(x_n)$. 
If $f$ was rejected by $\B$  as witnessed by $B\in\B$, this would contradict the continuity
of $F$ at $y_B$ as $\lim_{k\in B}x_k=y_B$ in $K_\B$ by Lemma
\ref{topological},
so by the property of
$\B$, there is a countable $\B_f\subseteq \B$ and $\lambda_f\in \K$ such that
$\B\setminus \B_f$ admits $f-\lambda_f 1_\N$. This means that 
$\lim_{k\in B}f(k)=\lambda_f$ for all $B\in \B\setminus \B_f$, and so for
$G=F-\lambda_f 1_{K_\B}$, we have
$$G(y_B)=(F-\lambda_f1_{K_\B})(y_B)=\lim_{k\in B}(F-\lambda_f1_{K_\B})(x_k)=
\lim_{k\in B}(f(k)-\lambda_f)=0$$ 
 for all but countably many $B\in \B$, which completes the proof  the lemma.
 \end{proof}

To construct $\B=\{B_\xi: \xi<\mathfrak c\}$ satisfying the hypothesis of Lemma \ref{reduction} and
consequently Theorem \ref{existence-F}, we need to make sure that 
the transfinite induction can be continued. There are many dangers along the way. One of them
is the well-known fact that it is consistent that there are maximal almost disjoint
families of cardinality less than $\mathfrak c$.   The main idea of how to avoid
these dangers is to obtain $\B$ as a simple modification of
an uncountable,  almost disjoint family which is Borel in the sense of Lemma \ref{borel-ad}.
Such a family is either countable or has  cardinality~$\mathfrak c$ by a powerful classical dichotomy due to Alexandrov and Hausdorff; see Lemma~\ref{borel-ch} for details.

Before we can execute this strategy, we require a few preparations.

\begin{lemma}[Borel set and function induced by $f\in\ell_\infty$]\label{borel-function}
 Let $f\in \ell_\infty$. Then
\[ D(f)=\{1_A\in \can: A \ \hbox{admits} \ f-\lambda1_\N\ \hbox{for some}\ \lambda\in \K\} \]
is a Borel  subset of~$\can$, and the function 
$L_f: D(f)\rightarrow \K$ given by $L_f(1_A)=\lim_{n\in A}f(n)$
is Borel measurable.
\end{lemma}

\begin{lemma}[A dichotomy for Borel functions]\label{Lemma11}
Suppose that $X$ is an uncountable Borel  subset of~$\can$ and  $\phi: X\rightarrow\K$ is Borel 
measurable. Then either $\phi$ is constant on a cocountable subset of $X$
or there are pairwise disjoint sets $\{x_\xi, x_\xi'\}\subseteq X$
for $\xi<\mathfrak c$ such that $\phi(x_\xi)\not=\phi(x_\xi')$ for every
$\xi<\mathfrak c$.
\end{lemma}
\begin{proof} By
Lemma \ref{borel-ch}, $X$ has cardinality $\mathfrak c$. 
Consider preimages of singletons under~$\phi$, which are Borel
subsets of $X$. 
If all of them are countable, we can easily choose the sets $\{x_\xi, x_\xi'\}$
by transfinite recursion. Otherwise $Y=\phi^{-1}[\{\lambda\}]$ has cardinality $\mathfrak c$
for some $\lambda\in \K$.
If $Y$ is a cocountable subset of $X$, we are  done. 
Otherwise $X\setminus Y$ has cardinality $\mathfrak c$ by Lemma \ref{borel-ch},
and so   we can choose $x_\xi\in Y$ and $x_\xi'\in X\setminus Y$ by transfinite recursion.
\end{proof}

Combining these two lemmas, we can  prove a preparatory dichotomy.

\begin{lemma}[Dichotomy for rejection and admission]\label{dichotomy-F} 
Suppose that $\A\subseteq\sub$ is an uncountable, almost disjoint family
such that $\{1_A: A\in \A\}$ is  a Borel subset of $\can$, and let $f\in \ell_\infty$.
Then one of the following holds:
\begin{enumerate}
\item either
$\A\setminus \A_f$ admits $f-\lambda_f 1_\N$ for some countable $\A_f\subseteq\A$ and $\lambda_f\in \K$,
\item or 
there are pairwise disjoint $\{A_\xi, A_\xi'\}\subseteq \A$ for all $\xi<\mathfrak c$ such that
$A_\xi\cup A_\xi'$ rejects $f$.
\end{enumerate}
\end{lemma}
\begin{proof}
Let $X=\{1_A: A\in \A\}\subseteq\can$, and suppose that $X\setminus D(f)$ is uncountable. Then
 $X\setminus D(f)$ has cardinality $\mathfrak c$ by Lemma \ref{borel-ch}.
If 
$1_{A}, 1_{A'} \in X\setminus D(f)$,
then  by the Decidability of Lemma \ref{monotonicity},  $A\cup A'$ rejects $f$,
so any pairwise disjoint family of two-element subsets of $X\setminus D(f)$ satisfies
the second alternative of the lemma.

Now suppose that $X\setminus D(f)$ is countable, which means that $X\cap D(f)$  has cardinality~$\mathfrak{c}$ by Lemma \ref{borel-ch}. Lemmas \ref{borel-function} and~\ref{Lemma11} yield
two possibilities for the function $\phi=L_f:D(f)\rightarrow \K$.
If $L_f$ is constant on a cocountable subset of $D(f)$, and thus of $X$,
we are in the first alternative of the lemma.
Otherwise there are
 pairwise disjoint sets
$\{1_{A_\xi}, 1_{A_\xi'}\}\subseteq X$ 
such that $L_f(1_{A_\xi})\not=L_f(1_{A_\xi'})$ for every $\xi<\mathfrak{c}$. Then the second
alternative of the lemma holds as $\lim_{n\in A_\xi\cup A_\xi'}f(n)$ does not exist.
\end{proof}

We are now ready to establish the main result:

\begin{theorem}[Existence]\label{existence-F} There is an uncountable, almost disjoint family \mbox{$\B\subseteq\sub$}
such that all bounded continuous functions $F:K_\B\rightarrow \K$ are of the form
$$F=\lambda 1_{K_\B}+ G,$$
 where $\lambda\in \K$ and $G$ is nonzero on at most countably
many points of $K_\B$. 
\end{theorem}
\begin{proof}
\begin{enumerate}
\item Using Lemma \ref{borel-ad},  fix an uncountable,
 almost disjoint family $\A\subseteq\sub$ such that $\{ 1_A: A\in \A\}$ is a Borel subset of~$\can$.
\item Let $\{f_\xi: \xi<\mathfrak c\}$ be an enumeration of the set of
  all elements $f\in \ell_\infty$ for which there is no countable
  $\A_f\subseteq \A$ and no $\lambda_f\in \K$ such that
  $\A\setminus\A_f$ admits $f-\lambda_f 1_\N$. (If this set is empty, then
  Lemma~\ref{reduction} shows that $\B = \A$ already has the required property; if it is nonempty, but has cardinality less
  than~$\mathfrak{c}$, simply repeat each element continuum many times.) Then the
  second alternative of the dichotomy for admission and rejection
  (Lemma~\ref{dichotomy-F}) holds for each $f_\xi$ and $\A$.
\item By transfinite recursion on $\xi<\mathfrak c$,
 construct $A_\xi, A_\xi'\in \A$ and $B_\xi\subseteq \N$ such that
 \begin{enumerate}
 \item $\{A_\eta, A_\eta'\}\cap \{A_\xi, A_\xi'\}=\emptyset$ for all $\eta<\xi<\mathfrak c$,
 \item $B_\xi=A_\xi\cup A_\xi'$ rejects $f_\xi$.
 \end{enumerate}
\item Define $\B=\{B_\xi: \xi<\mathfrak c\}$.
\item Check the required properties of $\B$, i.e., that the hypothesis
  of Lemma \ref{reduction} holds.  Take $f\in \ell_\infty$. If
  $f=f_\xi$ for some $\xi<\mathfrak c$, then $B_\xi$ rejects
  $f_\xi$. Otherwise $\A\setminus\A_f$ admits $f-\lambda_f 1_\N$
  for some countable $\A_f\subseteq \A$ and some $\lambda_f\in \K$, and then
  the Monotonicity of Lemma \ref{monotonicity} and con\-di\-tion~(b) above imply that $\B\setminus \B'$ 
   admits $f-\lambda_f 1_\N$ for some countable $\B'\subseteq \B$.\qedhere
\end{enumerate}
\end{proof}

\subsection{Constructing $C_0(K_\A)$ with few operators}\label{few-operators-section}

Our main construction follows the stages of the construction from Subsection~\ref{few-multiplications-section}. Note that the lemmas below will be proved in the 
following sections and keep their numbering from those sections.

Recall that we used elements  of $\ell_\infty$  to represent bounded continuous functions on~$K_\B$. Here we shall use $\K$-valued $\N\times\N$ matrices to represent operators on~$C_0(K_\B)$.
\begin{definition}\label{matrices} Let
$M=(m_{k, n})_{k, n\in \N}$ be an $\N\times\N$ matrix whose entries belong to~$\K$. Then:
\begin{itemize}
\item $\M=\{M: \|M\|<\infty\}$, where $\|M\| = \sup\{\sum_{n\in \N}|m_{k,n}|: k\in \N\}$.
\item For $f\in \ell_\infty$, 
$Mf\in\ell_\infty$ is given by $(Mf)(k)=\sum_{n\in \N}m_{k, n}f(n)$ for $k\in \N$.
\item $I$ stands for the $\N\times\N$ matrix which has entries~$1$ on the diagonal and all remaining entries are $0$.
\end{itemize}
\end{definition}

\begin{definition}\label{admits-M} Suppose that $M\in \M$, $\A\subseteq\sub$ and $B\in\sub$. We say that:
\begin{itemize}
\item $\A$ \emph{admits} $M$  
 if $\lim_{k\in A'}(M1_A)(k)=0$ for every $A, A'\in \A$.
\item $B$ \emph{rejects} $M$ if 
  $\lim_{k\in B}(M1_B)(k)$ does not exist.
\item  $B$ \emph{undermines} $M$ if there is  $n\in \N$
such that the sequence $((M1_{\{n\}})(k))_{k\in B}$ does not converge to zero.
\end{itemize}
\end{definition}

The elements of $\M$ are not as well suited to represent bounded linear operators
on~$C_0(K_\B)$ as the elements of $\ell_\infty$ were to represent  bounded continuous functions on~$K_\B$. However,  they do represent operators to some extent. We analyze this si\-tu\-a\-tion in Section~\ref{scattered} where, in addition to some general results about operators on $C_0(K)$ for $K$ scattered,  we obtain
the following analogue of Lemma \ref{reduction}.
\vskip 10pt
\begin{replemma}[\ref{reduction2}]{\rm(Reduction Lemma)\textbf{.}}
  Suppose that  $\B\subseteq\sub$ is an almost disjoint family such that,  for every $M\in \M$, one of the following three conditions holds:
  \begin{enumerate}
    \item there are uncountably many $B\in\B$ which reject $M$,
\item or there are uncountably many $B\in\B$ that undermine $M$,
\item or there is $\lambda\in \K$ and 
a countable subset $\B'\subseteq \B$ such that $\B\setminus \B'$ admits $M-\lambda I$.
\end{enumerate}
Then all bounded linear operators $T: C_0(K_\B)\rightarrow C_0(K_\B)$ are of the form
$$T=\lambda \Id+S,$$
 where $\lambda\in \K$ and $S$  is an operator which factors through $c_0$.
\end{replemma}
\vskip 10pt
Hence the rest of the efforts of the paper are focused on the construction of an almost
disjoint family which satisfies the hypothesis of Lemma \ref{reduction2}.
If we try to follow the ideas of Subsection~\ref{few-multiplications-section}
we quickly realize that the corresponding version of Lemma \ref{monotonicity}
fails badly for the simple reason that elements of~$\M$ 
need not  be monotone.
We do not know  how to overcome a number of problems stemming from this fact while working
with the above notion of admission.

Instead our approach is to consider a version of admission which we call
``acceptance'' that is sufficiently monotone and can be viewed as a kind of
``hereditary admission". This change, on the other hand,
complicates other parts of the construction. Our motivation comes from the theory 
of compact operators on~$c_0$ and~$\ell_\infty$ and is explained
 in more detail at the end of Section~\ref{matrices-section}.

 We introduce  the following notation and terminology for matrices:
\begin{definition}\label{compact-def} Suppose that $M=(m_{k, n})_{k, n\in \N}\in \M$.
\begin{itemize}
\item For $j\in\N$, set $M_j=(m_{k, n}')_{k, n\in \N}$, where $m_{k, n}'=0$ if $n\leq j$ and
$m_{k, n}'=m_{k, n}$ otherwise.
\item $M$ is called a  \emph{compact matrix} if $\lim_{j\in \N}\|M_j\|=0$.
\item For $A\in\sub$, define $M^A=(m_{k, n}')_{k, n\in \N}$, 
where $m_{k, n}'=m_{k, n}$ if $n, k\in A$ and $m_{k, n}'=0$ otherwise.
\item We say that $A\in\sub$ \emph{accepts} 
$M$ if $M^A$ is a compact matrix.
\item  We say that  $\A\subseteq\sub$
 \emph{accepts} $M$ if $A\cup A'$  accepts $M$ for every  $A, A'\in \A$.
\end{itemize}
\end{definition}

Section \ref{matrices-section} contains the discussion of the relevant operator theoretic aspects
of compact matrices and culminates in the proof of the following result:
\vskip 10pt
\begin{replemma}[\ref{acceptance}] Suppose that $M\in\M$ and that $\A\subseteq\sub$.
\begin{enumerate}
\item \emph{(Admission)} If $\A$  accepts $M$ and no element of $\A$  undermines $M$,
then $\A$ admits $M$.
\item \emph{(Monotonicity)} If $\A$  accepts $M$, then
$B$  accepts $M$ for every $B\in\sub$ which is included in a finite union
of elements of $\A$.
\item \emph{(Decidability)} If $A\in\sub$ does not  accept $M-\lambda I$ for any $\lambda\in \K$, 
then there is an infinite $B\subseteq A$ such that $B$ rejects $M$.
\item \emph{(Amalgamation)}
If,  for every $A, A'\in\A$, there is $\lambda_{A,A'}\in \K$ such that
$A\cup A'$  accepts $M-\lambda_{A,A'} I$, then there is $\lambda\in\K$
such that $\A$  accepts $M-\lambda I$.
\end{enumerate}
\end{replemma} 
\vskip 10pt

In Section \ref{borel-section} we embark on developing some tools which will enable us to take advantage
of the Borel structure of~$2^\N$, beginning with the following lemma. 
\vskip 10pt
\begin{replemma}[\ref{borel-ad}] There is an uncountable,
almost disjoint family $\A\subseteq\sub$ such that the set $\{1_A: A\in
\A\}\subseteq \can$ is closed.
\end{replemma}
\vskip 10pt
Dealing with operators and matrices brings us into the noncommutative world, where we must consider pairs of indices $A, A'\in \A$ (see, e.g.,
Definitions \ref{admits-M} and~\ref{compact-def}),  rather than single indices as in
Subsection~\ref{few-multiplications-section}. To handle this situation, we require a Borel dichotomy for pairs.
\vskip 10pt
\begin{replemma}[\ref{van-engelen}]{\rm (A dichotomy for Borel sets in the square)}
Suppose that $X\subseteq\can\times\can$ is a Borel set. Then 
one of the following  conditions holds:
\begin{enumerate}
\item either there is
a countable $Y\subseteq\can$ such that each point of $X$ 
has at least one of its coordinates in $Y$,
\item or there is $Z\subseteq X$ of cardinality continuum such that for any distinct 
points $(p, q), (p', q')\in Z$, we have $\{p, q\}\cap\{p', q'\}=\emptyset$.
\end{enumerate}
\end{replemma}
\vskip 10pt
\begin{replemma}[\ref{matrices-borel}] Let $M\in \M$. Then
$$E(M)=\{(1_A, 1_{A'})\in \can\times \can: A\cup A' \
\hbox{accepts} \ M-\lambda I \ \hbox{for some}\ \lambda\in\K\}$$
is a Borel subset of $\can\times\can$. 
\end{replemma}
\vskip 10pt
The above two lemmas and Lemma \ref{acceptance} allow us to prove the following
analogue of Lemma \ref{dichotomy-F}:
\vskip 10pt
\begin{replemma}[\ref{dichotomy-M}]{\rm (A dichotomy for acceptance and rejection)} 
Suppose that $M\in \M$ and
$\A\subseteq\sub$ 
is an uncountable, almost disjoint family
such that $\{1_A: A\in \A\}$ is  a Borel subset of $\can$.
Then one of the following holds:
\begin{enumerate}
\item either
$\A\setminus \A'$  accepts $M-\lambda I$ for some countable $\A'\subseteq\A$ and $\lambda\in \K$,
\item or 
there are pairwise disjoint sets $\{A_\xi, A_\xi'\}\subseteq \A$ and an infinite
subset $B_\xi\subseteq A_\xi\cup A_\xi'$ such that
$B_\xi$ rejects $M$  for every $\xi<\mathfrak c$.
\end{enumerate}
\end{replemma}
\vskip 10pt
This, in particular, means that in the construction of the almost disjoint family satisfying
the hypothesis of the Reduction Lemma \ref{reduction2}, we need to handle not only unions
of pairs as in the proof of Theorem \ref{existence-F}, but also their subsets.
Fortunately the Monotonicity of Lemma \ref{acceptance} shows that our notion of  acceptance  allows us to do that, emphasizing how acceptance is a kind of ``hereditary admission".

We are now ready to show how our main result can be deduced from the above lemmas.\vskip 10pt
\begin{reptheorem}[\ref{main-theorem}] 
 There is an uncountable, almost disjoint family
$\B\subseteq\sub$ such that every bounded linear operator $T\colon C_0(K_\B)\to C_0(K_\B)$ has the form
$T=\lambda \Id+S$, where $\lambda\in\K$  and $S\colon C_0(K_\B)\to C_0(K_\B)$  factors through
$c_0$. 
\end{reptheorem}

\begin{proof}
\begin{enumerate}
\item Using Lemma \ref{borel-ad},  fix an uncountable, almost disjoint 
family $\A\subseteq\sub$ such that $\{1_A: A\in \A\}$
is a Borel subset of $\can$. 
\item Let~$\M'$ be the set of all matrices $M\in\M$ such that there
  is no countable $\A_M\subseteq \A$ and no $\lambda_M\in \K$ such that
  $\A\setminus\A_M$ accepts $M-\lambda_M I$. If $\M'=\emptyset$, set $\B=\A$ and go straight to Step~(5). Otherwise let $\{M_\xi: \xi<\mathfrak c\}$ be an enumeration of~$\M'$ with each matrix repeated continuum many times, and note that the second alternative of the dichotomy for
  acceptance and rejection (Lemma \ref{dichotomy-M}) holds for each~$M_\xi$ and~$\A$.
\item By transfinite recursion on $\xi<\mathfrak c$, construct $A_\xi,
  A_\xi'\in \A$ and $B_\xi\in\sub$ such that
 \begin{enumerate}
 \item $\{A_\eta, A_\eta'\}\cap \{A_\xi, A_\xi'\}=\emptyset$ for all $\eta<\xi<\mathfrak c$,
 \item $B_\xi\subseteq A_\xi\cup A_\xi'$,
 \item $B_\xi$ rejects $M_\xi$.
 \end{enumerate}
\item Define $\B=\{B_\xi: \xi<\mathfrak c\}$.
\item We check the required properties of $\B$, i.e., that $\B$
  satisfies one of the three conditions stated in the Reduction Lemma
  \ref{reduction2}. Fix $M\in\M$.

 If
  $M\in\M'$, then there are continuum many $\xi<\mathfrak{c}$ such that $M_\xi=M$, 
  and $B_\xi$ rejects $M$ for each of these~$\xi$,
  so the first condition in
  Lemma~\ref{reduction2} holds.
 If the set $\B' = \{ B\in\B :
  B\ \text{undermines}\ M\}$ is uncountable, then the second condition in
  Lemma~\ref{reduction2} holds.

It remains to consider the case where the set $\B'$ is
  countable and $M\notin\M'$, so that  $\A\setminus\A_M$ accepts $M-\lambda_M I$ for some
  countable $\A_M\subseteq \A$ and some $\lambda_M\in \K$.
Note that no element of $\B\setminus\B'$ undermines $M-\lambda_M I$ 
because $(\lambda_M I)1_{\{n\}}=\lambda_M1_{\{n\}}$ for every $n\in\N$. 
  The
  Monotonicity of Lemma \ref{acceptance} implies that we can find a 
  countable set $\B''\subseteq\B$ such that $\B\setminus \B''$ accepts
  $M-\lambda_M I$ since all but countably many elements of $\B$ are
  subsets of unions of pairs of elements of $\A\setminus\A_M$.
 Hence every element of  $\B\setminus(\B'\cup\B'')$ accepts  \mbox{$M-\lambda_M I$} and no element of
$\B\setminus(\B'\cup\B'')$ undermines
$M-\lambda_M I$, so the Admission of Lemma~\ref{acceptance} implies that $\B\setminus(\B'\cup\B'')$
admits \mbox{$M-\lambda_M I$}. Therefore the third condition in Lemma~\ref{reduction2} holds,
and the theorem follows.\qedhere 
\end{enumerate}
\end{proof}

Hence we have proved our main  Theorem \ref{main-theorem} subject to
the proofs of Lemmas \ref{reduction2},  \ref{acceptance}, 
 \ref{borel-ad},  \ref{van-engelen}, \ref{matrices-borel},
\ref{dichotomy-M} which will be proved in the following sections.

\section{Bounded linear operators on $C_0(K)$ for $K$ Hausdorff, 
locally compact and scattered}\label{scattered}
\noindent
In this section we consider the following standard Banach spaces
over the scalar field~$\K$, where~$\Gamma$ is an arbitrary index set:
$\ell_\infty(\Gamma)$ consisting of all $f\colon\Gamma\to\K$ such that
\[ \lVert f\rVert_\infty := \sup_{\gamma\in\Gamma}\lvert
f(\gamma)\rvert<\infty \] and $\ell_1(\Gamma)$ consisting of all
$f\colon\Gamma\to\K$ such that $\lVert f\rVert_1 :=
\sum_{\gamma\in\Gamma}\lvert f(\gamma)\rvert<\infty$, as well
as~$c_0(\Gamma)$, which is the closure in~$\ell_\infty(\Gamma)$ of the
subspace $c_{00}(\Gamma)$ consisting of finitely supported elements.
We shall also consider the collection $\M(\Gamma)$ of $\K$-valued
$\Gamma\times \Gamma$ matrices $M=(m_{\gamma, \gamma'})_{\gamma,
  \gamma'\in \Gamma}$ such that $\sup\{\|(m_{\gamma,
  \gamma'})_{\gamma'\in\Gamma}\|_1: \gamma\in \Gamma\}<\infty$.  As
usual we multiply such a matrix~$M$ by an element $f\in \ell_\infty(\Gamma)$
to obtain the element $Mf\in \ell_\infty(\Gamma)$ defined by \[ (Mf)(\gamma)=\langle
(m_{\gamma, \gamma'})_{\gamma'\in \Gamma}, f\rangle = \sum_{\gamma'\in\Gamma} m_{\gamma,\gamma'}f(\gamma'). \]

\subsection{Representing operators by infinite matrices}
Recall that all topological spaces we consider are assumed to be
Hausdorff. A topological space is \emph{scattered} if every (closed)
nonempty subset of it has a relatively isolated point. For basic
properties of compact scattered spaces, see~8.5 of~\cite{semadeni}.
The Banach space $C(K)$ for~$K$ compact and scattered is characterized
by an impressive list of strong conditions, many of which were already
surveyed in~\cite{pelczynski-semadeni}. For instance, $K$ is scattered
if and only if~$C(K)$ is an Asplund space (see \cite{vector-measures}).
Another condition that is equivalent to $K$ being scattered is that
the dual space~$C(K)^*$ is isometric to~$\ell_1(K)$. This condition is
due to Rudin~\cite{rudin} and will play a key role in the
following.

If $K$ is locally compact and scattered,
then its one-point compactification~$\alpha K$ is compact and
scattered, and so $C(\alpha K)^*$ is isometric to $\ell_1(\alpha
K)$. By the Riesz representation theorem for locally compact spaces
(see, e.g., 18.4.1 of \cite{semadeni}), the space $C_0(K)^*$ is formed of
Radon measures on~$\alpha K$ which vanish outside of~$K$.  This means
that~$C_0(K)^*$ is isometric to $\ell_1(K)$, and hence a locally
compact variant of Rudin's theorem holds as well. We will identify the
elements of $C_0(K)^*$ with Radon measures on~$K$ and with
elements of~$\ell_1(K)$ (i.e., absolutely summable functions on~$K$).

For $T\in\BB(C_0(K))$ and $x\in K$, we have $T^*(\delta_x) = \sum_{y\in K}T^*(\delta_x)(\{y\})\delta_y$, where~$T^*$ stands for
the adjoint of $T$ and $\delta_x$ for the probability measure concentrated in~$x$, and hence  
\[ T(f)(x)=T^*(\delta_x)(f)= \sum_{y\in K}T^*(\delta_x)(\{y\})f(y) = \int fdT^*(\delta_x)=\langle T^*(\delta_x),  f \rangle \]
for $f\in C_0(K)$.
 This suggests looking at operators on $C_0(K)$ as matrices in~$\M(K)$.

\begin{definition} Suppose that $K$ is a scattered,  locally compact Hausdorff space 
and that $T: C_0(K)\rightarrow C_0(K)$ is
bounded and linear. Then the \emph{matrix} of $T$ is the $K\times K$ matrix $M_T=(m_{x, y})_{x, y\in K}$ given by
$m_{x,y}=T^*(\delta_x)(\{y\})$ for $x, y\in K$. 
\end{definition}

\begin{lemma}\label{operators-scattered} Suppose that $K$ is a scattered, locally compact Hausdorff space. 
Then for every bounded linear operator $T: C_0(K)\rightarrow C_0(K)$ we have 
$M_T\in \M(K)$, and 
$$T(f)= M_Tf$$
for every $f\in C_0(K)$.
\end{lemma}
\begin{proof}  We have $M_T\in \M(K)$ by Rudin's theorem mentioned above, and 
$T(f)(x)=\sum_{y\in K}T^*(\delta_x)(\{y\})f(y)=(M_Tf)(x)$ for every $f\in C_0(K)$ and
$x\in K$.
\end{proof}

As usual, $K'$ denotes the Cantor--Bendixson derivative of a topological space~$K$, that is, the subspace
of~$K$ formed of the nonisolated points of $K$.
If $K$ is scattered, then the set of isolated points $K\setminus K'$ is dense in $K$ (since
isolated points relative to an open set are isolated in $K$). 

\begin{definition}\label{reduced-matrix}
Let $K$ be a scattered, locally compact Hausdorff space.  The
\emph{reduced matrix} of a bounded linear operator $T:
C_0(K)\rightarrow C_0(K)$ is $M_T^r=(m_{x, y})_{x, y\in K\setminus K'}$, where $m_{x,y}=T^*(\delta_x)(\{y\})$ for $x, y\in K\setminus K'$, as above.
\end{definition}

Any continuous function is, of course, determined by its  values on any dense subset of its domain. Analogously, bounded linear operators on~$C_0(K)$ are, to some extent, determined by 
their reduced matrix.

\begin{lemma}\label{reduced-lemma}
Let $K$ be an infinite, scattered, locally
 compact Hausdorff space,  and let $D=K\setminus K'$ be the set of its isolated points.
Suppose that $T: C_0(K)\rightarrow C_0(K)$ is
bounded and linear. Then there is   $E\subseteq K'$ of cardinality 
not bigger than the cardinality of $D$  such that
$$T(f)|D=M_T^r(f|D)$$
whenever $f\in C_0(K)$ and $f(y)=0$ for every $y\in E$.
\end{lemma}
\begin{proof}
Let $E=\{y\in K': T^*(\delta_x)(\{y\})\not =0\ \text{for some}\ x\in
D\}.$ As $T^*(\delta_x)\in \ell_1(K)$, for a given $x\in K$, the value
$T^*(\delta_x)(\{y\})$ is nonzero for at most countably many $y\in
K$. Hence the cardinality of $E$ is not bigger than the cardinality of
$D$, as $D$ is infinite by the hypothesis that $K$ is infinite.  For
$x\in D$ and $f\in C_0(K)$ with $f(y)=0$ for all $y\in E$, we have
$$T(f)(x)=\sum_{y\in K\setminus K'} T^*(\delta_x)(\{y\}) f(y)=\big(M_T^r(f|D)\big)(x),$$
as required.
\end{proof}

A natural question is which $D\times D$ matrices are reduced matrices
of bounded linear operators on $C_0(K)$, or which $K\times K$ matrices
are matrices of bounded linear operators on $C_0(K)$.  This, of
course, depends on the topology of $K$, which de\-ter\-mines the space
$C_0(K)\subseteq\ell_\infty(K)$. However, we know that such matrices
belong to~$\M(D)$ and $\M(K)$, respectively.  Using standard arguments, one can show that these matrices correspond
exactly to all bounded linear operators from $c_0(D)$ in\-to~$\ell_\infty(D)$ (respectively, from $c_0(K)$ into $\ell_\infty(K)$),
and such operators are in isometric correspondence with the
weak$^*$-continuous operators on $\ell_\infty(D)$ (respectively, on~$\ell_\infty(K)$), and with the adjoints of operators on~$\ell_1(D)$
(respectively, on~$\ell_1(K)$). We refer to~\cite{wilansky} for details of these correspondences, which 
will not be exploited
here.\label{matricesasops}

\subsection{Bounded linear operators on the Banach space $C_0(K_\A)$}

The main purpose of this subsection is to prove the Reduction Lemma
\ref{reduction2}.  For this we need a piece of terminology and a
couple of lemmas. First recall the terminology and the topological
facts concerning the space $K_\A$ from Subsection~\ref{psi-section}.

\begin{definition}\label{Defn_sf} Suppose
that $\A\subseteq\sub$ is an almost disjoint family.
For $f\in C_0(K_\A)$ and $\mathcal{X}\subseteq C_0(K_\A)$, we define
$$s(f)=\{A\in \A: 
f(y_A)\not=0\},\quad s(\mathcal{X})=\bigcup\{s(f): f\in \mathcal{X}\}.$$
\end{definition}

\begin{lemma}\label{separability-lemma} Suppose
that $\A\subseteq\sub$ is an almost disjoint family.  A closed
sub\-space~$\mathcal{X}$ of $C_0(K_\A)$ is separable if and only if
$s(\mathcal{X})$ is countable.
\end{lemma}
\begin{proof}
First note  that for any $f\in C_0(K_\A)$, the set  $s(f)$ is countable
because only finite subsets of the discrete, closed set $\{y_A: A\in \A\}$
are compact, and $\{x\in K_\A: |f(x)|\geq\varepsilon\}$ is compact for each $\varepsilon>0$.

Now suppose that $\mathcal{X}$ is separable, and 
let $\mathcal D$ be a countable dense subset of $\mathcal{X}$. Then $\B=s(\mathcal D)$ is
countable. Note that $s(\X)\subseteq \B$ because 
 the set 
 $$\{f\in C(K_\A): s(f)\subseteq \B\}=\{f\in C(K_\A):
 f(y_A)=0\ \hbox{for all}\ A\in\A\setminus \B\}$$ is a closed
 subspace of $C_0(K_\A)$, and it contains $\mathcal D$, so it contains
 $\mathcal{X}$. Hence $s(\mathcal{X})$ is countable.

On the other hand, if $s(\X)$ is countable, then
$\X$ is isomorphic to a subspace
of $C_0(K_{s(\X)})$, which is separable by Lemma \ref{topological}.
\end{proof}

\begin{lemma}[Example~2c of \cite{johnson-lindenstrauss}, cf.\ Lemma~3 of \cite{mrowka}]\label{jl-c0}
 Suppose that  $\A\subseteq\sub$ is an almost disjoint family and that $\X\subseteq C_0(K_\A)$ is separable.
Then there is a closed sub\-space~$\mathcal{Y}$ of $C_0(K_\A)$ such that $\X\subseteq \mathcal{Y}$ and $\mathcal{Y}\sim c_0$.
\end{lemma}

Recall the terminology introduced in Subsection \ref{few-operators-section}, in particular 
Definitions \ref{matrices} and \ref{admits-M}.

\begin{lemma}[Reduction Lemma]\label{reduction2}
 Suppose that  $\B\subseteq\sub$ is an almost disjoint family such that,  for every $M\in \M$, one of the following three conditions holds:
  \begin{enumerate}
    \item there are uncountably many $B\in\B$ which reject $M$,
\item or there are uncountably many $B\in\B$ that undermine $M$,
\item or there is $\lambda\in \K$ and 
a countable subset $\B'\subseteq \B$ such that $\B\setminus \B'$ admits $M-\lambda I$.
\end{enumerate}
Then all bounded linear operators $T: C_0(K_\B)\rightarrow C_0(K_\B)$ are of the form
$$T=\lambda \Id+S,$$
 where $\lambda\in \K$ and $S$  is an operator which factors through $c_0$.
\end{lemma}
\begin{proof} Let $K=K_\B$, and
fix a bounded linear operator $T: C_0(K)\rightarrow C_0(K)$.
Consider its reduced matrix $M_T^r=(m_{x_k, x_n})_{x_k, x_n\in K\setminus K'}$, and define
$M=(m_{k, n})_{k, n\in \N}$
by $m_{k, n}=m_{x_k, x_n}$ for all $k, n\in \N$. 
Lemma \ref{operators-scattered}  and Definition \ref{reduced-matrix}
 imply
that  $M\in \M$.

First we shall see that for this choice of $M$, there are at most countably many elements \mbox{$B\in\B$} which reject or undermine~$M$. This will be
based on the part of Lemma~\ref{topological} which says that
$\lim_{k\in B}x_k=y_B$ for every $B\in \B$. In fact
 the continuity of $T(1_{U(B)})$ at all points $y_{B'}$ for
$B'\in \B$  will contradict 
rejecting or undermining~$M$ for un\-count\-ably many elements of~$\B$.

As $D=K\setminus K'=\{x_n : n\in \N\}$ is countable, by Lemma
\ref{reduced-lemma} there is a countable set $E\subseteq K'=\{y_B: B\in
\B\}$ such that $T(f)|D=M^r_T(f|D)$ whenever $f\in C_0(K)$ satisfies
$f(y)=0$ for every $y\in E$.  Let $\B_1=\{B\in\B : y_B\in E\}$. The above condition
translates into $T(1_{U(B)})|D=M^r_T(1_{U(B)}|D)$ for $B\in\B\setminus
\B_1$, which implies that
 $$T(1_{U(B)})(y_B)=\lim_{k\in B}T(1_{U(B)})(x_k)=\lim_{k\in
  B}\big(M^r_T(1_{U(B)}|D)\big)(x_k)= \lim_{k\in B}(M1_B)(k).$$ Hence
no element of $\B\setminus\B_1$ rejects $M$, so in particular at most
countably many elements of~$\B$ reject $M$.
 
The subspace $\X=\{f\in C_0(K): f(y_B)=0\ \hbox{for all}\ B\in \B\}$ is
separable because $\{1_{\{x_n\}}: n\in \N\}$ spans a dense subspace of
it. It follows that~$T[\X]$ is separable, and therefore $\B_2=s(T[\X])$ is
countable by Lemma \ref{separability-lemma}. Let $n,k\in\N$. As
$x_n\not\in K'$, Lemma~\ref{reduced-lemma} implies that
$T(1_{\{x_n\}})(x_k) = M_T^r(1_{\{x_n\}}|D)(x_k) = (M1_{\{n\}})(k)$, and
so $\lim_{k\in B}(M1_{\{n\}})(k) = T(1_{\{x_n\}})(y_B) = 0$ for every
$B\in \B\setminus \B_2$. This shows that no element of $\B\setminus
\B_2$ undermines $M$, so in particular at most countably many elements
of~$\B$ undermine~$M$.
 
Hence the hypothesis about $\B$ implies that there is $\lambda\in \K$ and a countable
sub\-set~$\B'\subseteq \B$ such that $\B\setminus\B'$ admits $M-\lambda I$. We shall now complete the proof by showing that the operator  $S=T-\lambda \Id\in\BB(C_0(K))$ factors through~$c_0$.   
Since \[ \Id(1_{U(B)})(x_k)=1_{U(B)}(x_k)=1_B(k) \] for $B\in\B$ and $k\in\N$, the condition that $\B\setminus\B'$ admits $M-\lambda I$ implies that
 $$0=\lim_{k\in B'}\big((M-\lambda I)1_B\big)(k)
=\lim_{k\in B'}\big((T-\lambda \Id)(1_{U(B)})\big)(x_k) = \lim_{k\in B'} S(1_{U(B)})(x_k)$$
for $B\in \B\setminus (\B_1\cup \B')$ and $B'\in \B\setminus \B'$. As $(x_k)_{k\in B'}$ converges to $y_{B'}$ and
$S(1_{U(B)})$ is continuous on $K$, this yields that 
$S(1_{U(B)})(y_{B'})=0$ for $B, B'$ as above, and so 
we have 
$$s\bigl(\{S(1_{U(B)}): B\in\B\setminus (\B_1\cup\B')\}\bigr)\subseteq \B'.$$
This, together with the fact that the characteristic functions $1_{U(B)}$  and $1_{\{x_n\}}$ for
$B\in \B$ and $n\in \N$ span a dense subspace
of~$C_0(K)$, implies that
$$s(S[C_0(K)])\subseteq \B'\cup s\bigl(\{S(1_{U(B)}): B\in \B_1\cup\B'\}\bigr)\cup 
s\bigl(\{ S(1_{\{x_n\}}) : n\in\N\}\bigr),$$
which is countable by Lemma \ref{separability-lemma}, and so
 $S[C_0(K)]$ is separable by another application of Lemma \ref{separability-lemma}.
 By Lemma~\ref{jl-c0}, the range of~$S$ is contained in a copy of $c_0$,
 so~$S$ factors through $c_0$, as required.
\end{proof}

\section{Compact $\N\times\N$ matrices}\label{matrices-section}
\noindent
The main purpose of this section is to prove Lemma~\ref{acceptance},
which is used in both Sub\-sec\-tion~\ref{few-operators-section} and Section~\ref{borel-section}.
Lemma~\ref{wc-diagonal} will also be used in Section~\ref{borel-section}.
The main motivation behind this section is of operator theoretic nature.
However, the results that we need are purely combinatorial, so
we leave their  operator theoretic aspects as a comment at the end of the section.
Recall Definitions~\ref{matrices} and~\ref{compact-def}.

\begin{lemma}\label{d-g} Suppose that $M=(m_{k, n})_{k, n\in \N}\in \M$. Then $M$ is compact if and only
  if $\lim_{n\rightarrow\infty}(\sum_{i\in F_n}m_{k_n, i})=0$ 
  whenever $(F_n)_{n\in \N}\subseteq\fin$ are pairwise disjoint
and $(k_n)_{n\in \N}\subseteq \N$ is strictly increasing.
\end{lemma}
\begin{proof} 
It is clear that if $M$ is compact, then the condition of the lemma
holds because $|\sum_{i\in F_n}m_{k_n, i}|\leq\| M_{\min(F_n)-1}\|$, and $\min(F_n)\to\infty$ as $n\to\infty$ since the sets~$F_n$ are pairwise disjoint.

Conversely, suppose that~$M$ is not compact. Then
there is $\varepsilon>0$
such that $\|M_j\|>\varepsilon$ for infinitely many $j\in \N$.
Using the definition of the norm, one can recursively construct
 pairwise disjoint sets $F_n'\in\fin$ such that there is $k_n\in \N$
 satisfying $\sum_{i\in F_n'}|m_{k_n, i}|>\varepsilon$ for each $n\in \N$. 
 Grouping the elements $i\in F_n'$ in four groups depending on the quadrant of the complex plane in which~$m_{k_n,i}$ lies, we can choose subsets $F_n\subseteq F_n'$ such  that $|\sum_{i\in F_n}m_{k_n, i}|>\varepsilon/(4\sqrt{2})$. 
 The same~$k_n$ cannot repeat itself 
 more than $4\sqrt{2}\|M\|/\varepsilon$ times, so we may assume that the sequence $(k_n)$ is strictly increasing. Thus the condition of the lemma fails.
\end{proof}

\begin{definition}\label{D(M)} Suppose that
$M=(m_{k, n})_{k, n\in \N}\in \M$.
Then  $D(M)= (m_{k, n}')_{k, n\in \N}$, where $m_{k,k}'=m_{k, k}$ for all $k\in \N$
and $m_{k, n}'=0$ whenever $k, n\in \N$ are distinct.
\end{definition}

\begin{lemma}\label{wc-diagonal}
 Suppose that $M=(m_{k,n})_{k, n\in \N}\in \M$ and $\lambda, \lambda'\in \K$. Then:
\begin{enumerate}
\item  $M$ is compact if and only
if $D(M)$ and $M-D(M)$ are compact. 
\item $D(M)$ is compact if and only if
$\lim_{n\in \N}m_{n, n}=0$.
\item If both $M-\lambda I$ and $M-\lambda'I$ are compact, then
$\lambda=\lambda'$.
\end{enumerate}
\end{lemma}
\begin{proof} 
Items (1) and (2)  follow directly from the definition of the compactness of a matrix.
 Under
the hypothesis of~(3), $(M-\lambda I)-(M-\lambda'I)=(\lambda'-\lambda)I$
is compact, and so $\lambda=\lambda'$ by (2).
\end{proof}

\begin{lemma}\label{half-killing} Suppose that  $M=(m_{k, n})_{k, n\in \N}\in \M$ is not compact.
 Then there is an infinite and coinfinite 
 $A\subseteq\N$ such that $\lim_{k\in\N}(M1_A)(k)$ does not exist. If additionally
 $D(M)$ is a compact matrix, then there is an infinite and coinfinite 
 $A\subseteq\N$ such  that 
 $\lim_{k\in \N\setminus A}(M1_A)(k)$ does not exist.
\end{lemma}
\begin{proof} 
First we will note that if the columns of $M$, that is, the sequences
$M1_{\{n\}}$, do not all converge, then we get the stronger conclusion
of the second part of the lemma.  Indeed, suppose that $M1_{\{n\}}$
does not converge for some $n\in\N$. Let $A\subseteq\N$ be infinite
and coinfinite such that $\lim_{k\in \N\setminus A} (M1_{\{n\}})(k)$
does not exist and $n\not\in A$. If $\lim_{k\in \N\setminus
  A}(M1_A)(k)$ does not exist, we are done.  Otherwise the set $A' = A\cup\{n\}$ has the required property because if $\lim_{k\in \N\setminus
  A}(M1_A)(k)$ and $\lim_{k\in \N\setminus
  A'}(M1_{A'})(k)$  both exist, then so does 
$$\lim_{k\in \N\setminus A'}(M1_{\{n\}})(k) = \lim_{k\in \N\setminus A'}(M1_{A'})(k) - \lim_{k\in \N\setminus A'}(M1_{A })(k),$$
which is a contradiction.

Next we will prove the lemma under the additional hypothesis that each
column of $M$ converges to $0$. By Lemma \ref{d-g} and the hypothesis
of the lemma, we can find pairwise disjoint sets $(F_n)_{n\in
  \N}\subseteq \fin$, a strictly increasing sequence $(k_n)_{n\in
  \N}\subseteq\N$ and $\varepsilon>0$ such that $|\sum_{i\in
  F_n}m_{k_n, i}|>\varepsilon$ for every $n\in \N$.  By recursion on
$j\in\N$, we can construct a strictly increasing sequence $(n_j)_{j\in
  \N}\subseteq \N$ such that
\begin{enumerate}
\item $\sum_{j'<j}\sum_{i\in F_{n_{j'}}} \lvert m_{k_{n_j}, i}\rvert<\frac{\varepsilon}{5}$ for every $j\in\N$, and
\item $\lvert\sum_{i\in F_{n_j}} m_{k_{n_{j'}}, i}\rvert<\frac{\varepsilon}{5\cdot2^j}$ for every $j'<j$.
\end{enumerate}
At stage $j$ of the recursion, use the fact that the columns of $M$
converge to $0$ to sa\-tis\-fy~(1), and use the fact that the sets~$F_n$
are pairwise disjoint and the rows of~$M$ are absolutely summable to
satisfy~(2). Let $A=\bigcup_{j\in \N} F_{n_{2j}}$.  By the choice of
the sets~$F_n$ and (1) and (2), we have $|(M1_A)(k_{n_{2j}})|>
\varepsilon-2\varepsilon/5$ and $|(M1_A)(k_{n_{2j+1}})|<
2\varepsilon/5$ for every $j\in \N$. Hence $M1_A$ does not converge.

Now to prove the second part of the lemma under the additional
hypothesis that each column converges to $0$, first note that the
hypothesis of the second part means that $\lim_{k\in \N} m_{k, k}=0$
by Lemma \ref{wc-diagonal}. In this case the recursive
construction above can be modified to yield an infinite and coinfinite
set $A\subseteq\N$ such that $\lim_{k\in \N\setminus A}(M1_A)(k)$ does
not exist by replacing the set~$F_n$ with
$F_n'=F_n\setminus\{k_n\}$ for every $n\in\N$ in the recursion requirements~(1) and~(2) and adding two more conditions:
\begin{enumerate}
  \setcounter{enumi}{2}
\item $\{k_{n_{j'}}: j'<j\}\cap F_{n_j}'=\emptyset$, and 
\item $k_{n_j}\notin \bigcup_{j'<j}F_{n_{j'}}'$ for every $j\in\N$. 
\end{enumerate}
The possibility of satisfying the modified (1) and (2) follows from
the hypothesis that $\lim_{k\in \N} m_{k, k}=0$, while (3) and (4) can
be satisfied by choosing~$n_j$ sufficiently large, using that the
sets~$F_n$ are pairwise disjoint. With such a modified construction we
obtain $\{k_{n_j}: j\in \N\}\subseteq \N\setminus A$, and so $\lim_{k\in
  \N\setminus A}(M1_A)(k)$ does not exist.

Finally let us prove the first part of the lemma in the general case
when all columns of $M$ converge, but not necessarily to $0$.  Let
$\lambda_n=\lim_{k\in \N}(M1_{\{n\}})(k)$, and
define $M'=(m_{k, n}')_{k, n\in \N}$ by $m_{k, n}'=\lambda_n$ for every $k,n\in\N$. Note that \mbox{$M'\in \M$} as $\sum_{n<i}|\lambda_n|$ can be
approximated by $\sum_{n<i}|m_{k, n}|\leq \|M\|<\infty$ for any
$i\in\N$ and sufficiently large $k\in \N$. Hence we have $M-M'\in\M$ as
well. Moreover, the summability of $\sum_{n\in\N}|\lambda_n|$ implies
that~$M'$ is compact by the definition of com\-pact\-ness.
We may now apply the previous case, where the columns
of the matrix converge to~$0$, to the matrix $M''=M-M'$, which is noncompact since $M$ is noncompact and $M'$ is compact. Hence  there exists 
 an infinite and coinfinite  set
 $A\subseteq\N$ such that $\lim_{k\in \N}(M''1_A)(k)$ does not exist.
The definition of $M'$ implies  that $M'f$ is a constant sequence,
 and thus convergent, for every $f\in \ell_\infty$. Therefore, 
 using that $(M1_A)(k)=(M'1_A)(k)+(M''1_A)(k)$, we conclude that $\lim_{k\in \N}(M1_A)(k)$ does not exist, as required.
 The argument for the second part of the lemma is analogous.
 One needs to note that $D(M'')$ is compact if $D(M)$ is.
 This follows from the above fact that
 $M'$ is compact, and so $D(M')$ is compact by Lemma \ref{wc-diagonal}.
\end{proof}

Recall  Definitions~\ref{matrices}, \ref{admits-M} and~\ref{compact-def}. 

\begin{lemma}\label{acceptance}
Suppose that $M\in\M$ and that $\A\subseteq\sub$.
\begin{enumerate}
\item \emph{(Admission)} If $\A$  accepts $M$ and no element of $\A$  undermines $M$,
then $\A$ admits $M$.
\item \emph{(Monotonicity)} If $\A$  accepts $M$, then
$B$  accepts $M$ for every $B\in\sub$ which is included in a finite union
of elements of $\A$.
\item \emph{(Decidability)} If $A\in\sub$ does not  accept $M-\lambda I$ for any $\lambda\in \K$, 
then there is an infinite $B\subseteq A$ such that $B$ rejects $M$.
\item \emph{(Amalgamation)}
If, for every $A, A'\in\A$, there is $\lambda_{A,A'}\in \K$ such that
$A\cup A'$  accepts $M-\lambda_{A,A'} I$, then there is $\lambda\in\K$
such that $\A$  accepts $M-\lambda I$.
\end{enumerate}
\end{lemma} 
\begin{proof}
(Admission) Fix $A, A'\in\A$. For every $f\in\ell_\infty$
and  $j\in \N$, we have
\[ (M^{A\cup A'}-M_j^{A\cup A'})f=\sum_{n\leq j}f(n)M^{A\cup A'}1_{\{n\}}, \]
which belongs to~$c_0$ 
because
the hypothesis that $A$ and $A'$ do not undermine $M$ implies that
$\lim_{k\in A\cup A'}(M1_{\{n\}})(k)=0$ for every $n\in\N$.
Since
$$M^{A\cup A'}f=(M^{A\cup A'}-M_j^{A\cup A'})f+M_j^{A\cup A'}f,$$
where $\|M^{A\cup A'}_jf\|\leq\|M^{A\cup A'}_j\|\|f\|\to 0$ as $j\to\infty$ by the compactness of $M^{A\cup A'}$,
we deduce that $M^{A\cup A'}f$ can be approximated in the supremum norm
by elements of~$c_0$.  This shows that
 $M^{A\cup A'}f\in c_0$ because~$c_0$ is a closed subspace of $\ell_\infty$.
In particular, choosing $f=1_A$, we have
$$0=\lim_{k\in A'}(M^{A\cup A'}1_A)(k)=\lim_{k\in A'}(M1_A)(k).$$
Since $A$ and $A'$ were arbitrary, this  proves that $\A$ admits $M$.
 
\noindent (Monotonicity) Let $A_1,\ldots, A_l\in\A$. The compactness of each matrix $M^{A_i\cup A_{i'}}$,
which follows from the hypothesis that $\A$ accepts $M$, means that
$$\lim_{j\in\N}\bigr(\sup_{k\in A_i}\sum_{n\in A_{i'},\, n>j} |m_{k, n}|\bigr)=0$$
 for every $1\leq i,i'\leq l$.
But this implies that, for every $B\subseteq A_1\cup\cdots\cup A_l$, we have
$$\lim_{j\in\N}\big(\sup_{k\in B}\sum_{n\in B,\, n>j} |m_{k, n}|\big)=0,$$
which shows that $M^B$ is compact, as required.

\noindent (Decidability) First assume that $M^A-D(M^A)$ is not a compact matrix.  Then, since
$D(M^A-D(M^A))$ is the zero matrix, we can apply a version of the second part
of Lemma \ref{half-killing} with $\N$ replaced with~$A$
 to find an infinite and coinfinite set $B'\subseteq A$ such that
$$\lim_{k\in A\setminus B'}\bigl((M^A-D(M^A))1_{B'}\bigr)(k)=
\lim_{k\in A\setminus B'}(M^A1_{B'})(k)$$ 
does not exist. If 
$\lim_{k\in A\setminus B'}(M^A1_{A\setminus B'})(k)$ 
does not exist, the conclusion follows by taking $B=A\setminus B'$. Otherwise, as $1_{B'}+1_{A\setminus B'}=1_A$, we
conclude that
$\lim_{n\in A\setminus B'}(M^A1_A)(n)$ does not exist, which implies
that $\lim_{n\in A}(M^A1_A)(n)$ does not exist, i.e., $B=A$ rejects $M$.

Now suppose that $M^A-D(M^A)$ is compact.
First we note that $(m_{k, k})_{k\in A}$ does not converge.
If it converged to $\lambda\in \K$, then
$D(M^A)-\lambda I^A$ would be compact by Lemma \ref{wc-diagonal}(2)
and so the identity
$$(M-\lambda I)^A= (M^A-D(M^A))+(D(M^A)-\lambda I^A)$$
would contradict the hypothesis that $A$ does not accept $M-\lambda I$ for
any $\lambda\in \K$.

Hence we can find two infinite, disjoint subsets $A_1$ and $A_2$ of~$A$
such that $(m_{k, k})_{k\in A_1}$ and $(m_{k, k})_{k\in A_2}$ converge
to distinct limits $\lambda_1$ and~$\lambda_2$, respectively. Let
$\varepsilon=|\lambda_1-\lambda_2|>0$.  Since $M^A-D(M^A)$ is compact,
there is $j\in \N$ such that $\|(M^A-D(M^A))_j\|<\varepsilon/3$.  Let
$B=(A_1\cup A_2)\setminus\{0,\ldots, j\}$.  Then
$$|(M1_B)(k)-m_{k, k}|=|\big((M^A-D(M^A))1_{B}\big)(k)|<\varepsilon/3$$ 
for every $k\in B$, so $\lim_{k\in B}(M1_B)(k)$ does not exist, as required. 

\noindent (Amalgamation) We need to find $\lambda\in\K$ such that $(M-\lambda I)^{A\cup A'}$
 is compact when\-ever $A, A'\in \A$. For every pair $A,A'\in\A$, we can choose $\lambda_{A, A'}\in \K$
  such that $(M-\lambda_{A, A'} I)^{A\cup A'}$ is compact  by the hypothesis.
 Monotonicity implies that \mbox{$(M-\lambda_{A, A'}I)^A$} and \mbox{$(M-\lambda_{A, A''}I)^A$} are
 compact for any $A, A', A''\in \A$, so $\lambda_{A, A'}=\lambda_{A, A''}$ by
  Lemma \ref{wc-diagonal}(3).  Hence $\lambda_{A, A'}=\lambda_{A, A''}=\lambda_{A''', A''}$
  for any $A, A', A'', A'''\in \A$, so all of these numbers have the same value $\lambda$,
  as required.
\end{proof}

Many of the results of this section 
can be interpreted in the language of bounded linear operators on Banach spaces rather than matrices. In fact
this was our original motivation. As we already mentioned on page~\pageref{matricesasops}, elements of $\M$ are exactly the matrices which
represent bounded linear operators from $c_0$ into $\ell_\infty$, or the adjoints
of operators in $\BB(\ell_1)$, that is, the
weak$^*$-continuous operators on~$\ell_\infty$. Here and below we use the dualities $c_0^*=\ell_1$ and
$\ell_1^*=\ell_\infty$. Compact matrices define exactly compact operators
from $c_0$ into $\ell_\infty$.
Using the Schur property of $\ell_1$, one can easily show that an operator
defined on $c_0$ is compact if an only if it is weakly compact. 

The condition of Lemma \ref{d-g} is the classical
Dieudonn\'{e}--Grothendieck characterization of weakly compact subsets
of the dual of a $C(K)$-space (see, e.g., Sec\-tion~5.3
of~\cite{albiac-kalton}).  Lemma \ref{half-killing} and the
Decidability of Lemma \ref{acceptance} are more specific versions of
the result~\cite{gdiestel} which says that a non-weakly compact
operator on a Grothendieck space ($\ell_\infty$~in this case) cannot
have separable range, so in particular its range cannot be contained
in the space~$c$ of convergent sequences. The Admission of
Lemma~\ref{acceptance} corresponds to the result which says that
$T^{**}[\mathcal{X}^{**}]\subseteq \mathcal{X}$ for any weakly compact
operator~$T$ on a Banach space $\mathcal{X}$ (see, e.g., Appendix~G
of~\cite{albiac-kalton}).  Here $\mathcal{X}=c_0$ and
$\mathcal{X}^{**}=\ell_\infty$, and one needs to note that applying the
same matrix from $\M$ to elements of $c_0$ and elements of
$\ell_\infty$ corresponds to the operator~$T$ and its second
adjoint~$T^{**}$, respectively. The Monotonicity of Lemma
\ref{acceptance} corresponds to the fact that (weakly) compact
operators form an ideal.

We opted for giving direct proofs in terms of $\N\times\N$ matrices
because in any case Lemma \ref{wc-diagonal} and the concrete
conditions of Lemma \ref{acceptance} require us to work with the
combinatorial structure of matrices.  Also the above mentioned
operator theoretic results require quite substantial abstract
preparations before they can be applied in our setting. This is not
necessary as the facts we have proved above are fairly elementary and,
as we have seen, can essentially be deduced from elementary properties
of the summability of infinite series.

\section{Borel structure, matrices and almost disjoint families}\label{borel-section}
\noindent
In this section we link $\K$-valued $\N\times\N$ matrices and almost
disjoint families in $\sub$ with the Borel structure of $\can$ and
$\can\times \can$. Here $\can$ is considered with the product
topology, so in particular the sets of the form $[s]=\{p\in \can: s\subseteq p\}$,
where $s$ is a finite partial function from $\N$ into $2=\{0,1\}$,
form a basis for the topology of~$\can$ consisting of clopen sets. 

Our aim is to transform the following cardinality dichotomy for
Borel sets due to Alexandrov and Hausdorff into the dichotomy for acceptance and rejection (Lemma~\ref{dichotomy-M}).

\begin{lemma}[A dichotomy for Borel sets]\label{borel-ch} If
$X\subseteq \can$ is Borel, then either $X$ is countable or $X$ has cardinality $\mathfrak c$.
\end{lemma}
\begin{proof} See 18.6 of \cite{kechris}.
\end{proof} 

Lemma~\ref{dichotomy-M} will apply to Borel
almost disjoint families, so first we need the following result:

\begin{lemma}\label{borel-ad} There is an uncountable,
almost disjoint family $\A\subseteq\sub$ such that the set $\{1_A: A\in \A\}\subseteq \can$ is closed.
\end{lemma}
\begin{proof} Let $\seq=\bigcup_{n\in\N} 2^n$, and
let $\phi: \N\rightarrow \seq$ be a bijection.  Given $p\in \can$,
define $A_p=\{p|n: n\in \N\}\subseteq\seq$, where $p|n\in 2^n$ is the
restriction of $p$ to $\{0,\ldots, n-1\}$. It is clear that $\{A_p:
p\in \can\}\subseteq[\seq]^\omega$ is an almost disjoint family of
cardinality $\mathfrak c$.

Let $\A=\{\phi^{-1}[A_p]: p\in \can\}$. We
shall prove that $\{1_A: A\in \A\}$ is closed in~$\can$ by showing that its complement is open. Let $1_A\in\can\setminus\{1_B: B\in \A\}$. Then
$\phi[A]$ is not of the form $A_p$ for any $p\in \can$.  There are
three possible reasons for this:
\begin{enumerate}
  \item There are $s, t\in \phi[A]$ such
    that $s\not\subseteq t$ and $t\not\subseteq s$.
  \item There are $s,
t\in\seq$ such that $s\subseteq t$ and $t\in \phi[A]$, but $s\not\in
\phi[A]$.
\item $A$ is finite, and so there is $n \in \N$ such that
  $\phi[A]\cap 2^n=\emptyset$.
\end{enumerate}
  All these cases define clopen
neighbourhoods of $1_A$ disjoint from~$\{1_B: B\in \A\}$, as required.
 \end{proof}
 
 Now we step up the cardinality dichotomy for Borel subsets of $\can$ to
 its square $\can\times\can$, using a result from~\cite{two-remarks} of  van Engelen, Kunen and Miller.

\begin{lemma}\label{van-engelen} Suppose that $X\subseteq\can\times\can$ is a Borel set. Then 
one of the following conditions holds:
\begin{enumerate}
\item either there is
a countable $Y\subseteq\can$ such that each point of $X$ 
has at least one of its coordinates in $Y$,
\item or there is $Z\subseteq X$ of cardinality continuum such that for any distinct 
points $(p, q), (p', q')\in Z$, we have $\{p, q\}\cap\{p', q'\}=\emptyset$.
\end{enumerate}
\end{lemma}
\begin{proof}  As $\can$ can be embedded as a closed subset of~$\R$, it is enough to prove the lemma for Borel subsets~$X$ of $\R^2$.

  First note that if there is a nonzero $a\in \R$ such that
$X$ intersects the line~$\ell$ given by $y=ax$ in an uncountable set, then the second alternative of the lemma is satisfied. Indeed, the intersection
  of~$\ell$ with $X$ is Borel and uncountable, and so of
  cardinality continuum by Lemma~\ref{borel-ch}. One can now construct the set~$Z$ by an easy transfinite
  recursion of length continuum by choosing its 
  elements  from this intersection, using the fact that given a point
  $(p,q)$ on~$\ell$, there may be at most two other points
  $(p',q')$ on~$\ell$ (namely $(p/a,p)$ and $(q,aq)$) with $\{p,
  q\}\cap\{p', q'\}\not=\emptyset$.

The above  observation implies the lemma in the case where~$X$ is  covered by countably many lines.
Indeed,  if there is a countable collection~$\mathcal{L}$ of lines covering~$X$, 
then $X=H\cup V\cup S$, where $H$, $V$ and $S$ are points on horizontal, vertical and sloping lines in~$\mathcal{L}$, respectively. If $S$ is uncountable, we are in the second alternative of the lemma by the preceding
observation. Otherwise $S$ is countable, so we can assume that $S=\emptyset$, as countably
many points can be covered by countably many horizontal and/or vertical lines.  Then the first alternative of the lemma is satisfied by the set $Y=\pi_y[H]\cup\pi_x[V]$, where
$\pi_x$ and $\pi_y$ are the projections on the $x$-axis and $y$-axis, respectively.

Hence we are left with the case where $X$ cannot be covered by countably many lines.
Then we can use a theorem of van Engelen, Kunen and Miller~\cite{two-remarks}, which says that
an analytic (in particular Borel) subset  of the plane which cannot be covered by countably many
lines contains a perfect subset $P$ (hence of size continuum) such that no three points in~$P$ are collinear. 
One may now construct a set~$Z\subseteq X$ such that the second alternative of the lemma is satisfied  by  choosing its elements from $P$ by an easy transfinite
recursion of length continuum, using the fact that given a point $(p,q)$ in $P$,
there are at most  six other points $(p', q')$  in $P$  with 
$\{p, q\}\cap\{p', q'\}\ne\emptyset$,
namely at most one point on  each of the lines
$x=p$ and $y=q$, and at most two points on  each of the lines $x=q$ and $y=p$.
\end{proof}

Before we can establish  our desired dichotomy for  acceptance and rejection (Lemma~\ref{dichotomy-M}),
we need to verify that certain sets induced by $\N\times\N$ matrices are Borel. This is done in
the following lemmas which culminate in Lemma \ref{matrices-borel}.

\begin{lemma}\label{borel-c} Suppose that $M=(m_{k, n})_{k, n\in \N}\in\M$. 
Then 
$$C(M)=\{1_A\in \can: M^A  \ \hbox{is  a compact matrix}\}$$
is a Borel subset of $2^\N$.
\end{lemma}
\begin{proof}
Note that for every $i, j,  k\in \N$, the set
$$S_{i, j,  k}(M)=\Bigl\{1_A\in \can: 
k\in A\ \Rightarrow \ \sum_{n\in A, n>j}|m_{k, n}|\leq \frac{1}{i+1}\Bigr\}$$
is closed in $2^\N$ because if $1_A\not \in S_{i, j, k}$, then
$[1_A|n]\cap S_{i, j, k}(M)=\emptyset$ for some $n\in \N$.
By the definition of  a compact matrix, we have 
$$C(M)=
\bigcap_{i\in \N}\bigcup_{j\in \N}\bigcap_{k\in \N}S_{i, j, k}(M),$$
so it follows that $C(M)$ is a Borel subset of $ \can$, as required.
\end{proof}

In the next two proofs, we use the notation
$(n, i)$ for $n\in \N$ and
$i\in \{0,1\}$ to denote the partial function  which has
do\-main~$\{n\}$ and takes the value~$i$ at~$n$.

\begin{lemma}\label{borel-conv} Suppose that $M=(m_{k, n})_{k, n\in \N}\in \M$.
Then 
$$\operatorname{Conv}(M)=\{1_A\in\can : (m_{n, n})_{n\in A}\  \hbox{converges}\}$$
is a Borel subset of $\can$.
\end{lemma}
\begin{proof}
For every $i,  n, n'\in\N$, the set 
$$\operatorname{Conv}_{i,  n, n'}(M)=\Bigl\{1_A\in 2^\N:  n, n'\in A\ \Rightarrow \ |m_{n, n}-m_{n', n'}|<\frac{1}{i+1}\Bigr\}$$
is clopen in $2^\N$ because it is equal to $\can$ if
$|m_{n, n}-m_{n', n'}|<1/(i+1)$, and it is equal to $[(n,0)]\cup[(n',0 )]$ otherwise.
Hence the conclusion follows from the fact that 
\[ \operatorname{Conv}(M)=\bigcap_{i\in \N}\bigcup_{j\in \N}\bigcap_{n, n'>j}\operatorname{Conv}_{i,  n, n'}(M).\qedhere \]
\end{proof} 

\begin{lemma}\label{matrices-borel} Let $M\in \M$. Then
$$E(M)=\{(1_A, 1_{A'})\in \can\times \can: A\cup A' \
\hbox{accepts} \ M-\lambda I \ \hbox{for some}\ \lambda\in\K\}$$
is a Borel subset of $\can\times\can$. 
\end{lemma}

\begin{proof} By Lemma \ref{wc-diagonal},  $(1_A, 1_{A'})\in E(M)$ if and only if 
 $1_{A\cup A'}\in C(M-D(M))$
and  $1_{A\cup A'}\in \operatorname{Conv}(M)$.
Hence
$$E(M)=\phi^{-1}[C(M-D(M))]\cap \phi^{-1}[\operatorname{Conv}(M)],$$
where  $\phi: \can\times\can\rightarrow \can$ is 
the function defined by $\phi(1_A, 1_{A'})=1_{A\cup A'}$. 
To check the continuity of $\phi$, it is enough to note that,  for any $n\in \N$,
\[ \phi^{-1}[[(n, 0)]]=[(n, 0)]\times[(n, 0)],\qquad
\phi^{-1}[[(n, 1)]]=\can\times [(n, 1)]\cup
   [(n, 1)]\times\can. \] 
   Therefore Lemmas \ref{borel-c} and  \ref{borel-conv} 
imply that $E(M)$ is a Borel subset of $\can\times\can$.
\end{proof}

\begin{lemma}[A dichotomy for acceptance and rejection]\label{dichotomy-M} 
Suppose that $M\in \M$ and
$\A\subseteq\sub$ 
is an uncountable, almost disjoint family
such that $\{1_A: A\in \A\}$ is  a Borel subset of $\can$.
Then one of the following holds:
\begin{enumerate}
\item either
$\A\setminus\A'$  accepts $M-\lambda I$ for some countable $\A'\subseteq\A$ and $\lambda\in \K$,
\item or 
there are pairwise disjoint sets $\{A_\xi, A_\xi'\}\subseteq \A$ and an infinite subset
$B_\xi\subseteq A_\xi\cup A_\xi'$ such that
$B_\xi$ rejects $M$  for every $\xi<\mathfrak c$.
\end{enumerate}
\end{lemma}
\begin{proof} Let $X(\A)=\{1_A: A\in \A\}$. 
By Lemma \ref{matrices-borel}, the set  
$$X=(X(\A)\times X(\A))\setminus E(M)$$
 is Borel as the Cartesian product
 of two Borel sets is Borel, where we recall that
$$E(M)=\{(1_A, 1_{A'})\in\can\times\can: A\cup A'\ \hbox{accepts}\
 M-\lambda I \ \hbox{for some}\ \lambda\in \K\}.$$
 We can therefore apply Lemma~\ref{van-engelen} to~$X$.

The first alternative of Lemma \ref{van-engelen} gives that there
is a countable $\A'\subseteq \A$ such that
if $A, A'\in \A\setminus \A'$, then
$(1_A, 1_{A'})\in E(M)$, that is, 
 $A\cup A'$ accepts $M-\lambda I$ for some $\lambda\in \K$. By the Amalgamation  
of Lemma \ref{acceptance}, this means that 
there is a single $\lambda\in \K$ such that 
$\A\setminus \A'$  accepts $M-\lambda I$.

The second alternative of Lemma \ref{van-engelen} gives a set
$Z\subseteq X$ of cardinality continuum such that
$\{1_A, 1_{A'}\}\cap\{1_B,
1_{B'}\}=\emptyset$ for any distinct
points $(1_A,1_{A'}), (1_B, 1_{B'})\in Z$. This yields a pairwise disjoint family $\{\{A_\xi,
A_{\xi}'\}: \xi<\mathfrak c\}$ with $A_\xi, A_{\xi}'\in\A$ such that
${A_\xi\cup A_\xi'}$ does not accept $M-\lambda I$ for any
$\lambda\in\K$ and any $\xi<\mathfrak c$.  By the Decidability of
Lemma \ref{acceptance}, this produces an infinite set $B_\xi\subseteq
A_\xi\cup A_\xi'$ such that $B_\xi$ rejects~$M$.
\end{proof}

\section{Consequences of
Theorem \ref{main-theorem}, context and open questions}\label{final}
\noindent
In this section we shall establish some consequences 
of Theorem~\ref{main-theorem}, first for the compact space~$\alpha K_\A$ and then for the Banach algebra $\BB(C_0(K_\A))$ of all bounded linear operators on~$C_0(K_\A)$.
Moreover, we shall discuss
 the place of the
 Banach space~$C_0(K_\A)$ 
 among other known Banach spaces with few operators, and we shall present
  and motivate some
 open questions which arise naturally from our work.

\subsection{Continuous self-maps of~$\alpha K_\A$ when $C_0(K_\A)$ has few operators}

\begin{proposition}\label{maps}
Suppose that~$\A\subseteq\sub$ is an uncountable, almost disjoint
family such that all bounded linear operators on~$C_0(K_\A)$ are of the form described in
Theorem~\ref{main-theorem}.  Then, for every continuous map $\phi:
\alpha K_\A\rightarrow \alpha K_\A$, either~$\phi$ has countable range
or the set of fixed points of $\phi$ is cocountable.
\end{proposition}
\begin{proof}

As $C_0(K_\A)$ is isomorphic to $C(\alpha K_\A)$, the hypothesis implies that
every bounded linear operator on $C(\alpha K_\A)$ is the  sum of a scalar
multiple of the identity and an operator with separable range. Given a  continuous map $\phi:
\alpha K_\A\rightarrow \alpha K_\A$, we will apply this hypothesis to the composition operator
$T_\phi: C(\alpha K_\A)\rightarrow C(\alpha K_\A)$ defined by $T_\phi(f)=f\circ \phi$
for $f\in C(\alpha K_\A)$.

We begin by showing that if $T_\phi$ has separable range, then $\phi$ has countable range. To this end, suppose that~$\phi$ has uncountable range.
Then there are distinct $B_\xi\in \A$ for $\xi<\omega_1$ such
that $y_{B_\xi}=\phi(y_{A_\xi})$ for some $A_\xi\in \A$.
For $\xi<\eta<\omega_1$, we have
\begin{align*} \|T_\phi(1_{U(B_\xi)})-T_\phi(1_{U(B_\eta)})\|&\geq
|T_\phi(1_{U(B_\xi)})(y_{A_\xi})-T_\phi(1_{U(B_\eta)})(y_{A_\xi})|\\
&=|1_{U(B_\xi)}(\phi(y_{A_\xi}))-1_{U(B_\eta)}(\phi(y_{A_\xi}))|\\ &=|1_{U(B_\xi)}(y_{B_\xi})-1_{U(B_\eta)}(y_{B_\xi})|=1-0=1. \end{align*}
This shows that the range
of~$T_\phi$ is nonseparable, and the conclusion follows.

We are now ready to prove the dichotomy stated in the
proposition. More pre\-cise\-ly, we shall show that if the set of fixed points of~$\phi$ is not
cocountable, then~$T_\phi$ has separable range, because it 
will then follow from the previous paragraph that~$\phi$ has countable range.
As explained in the first paragraph, the hypothesis of the proposition implies that there is a scalar~$\lambda$ such that $T_\phi-\lambda\operatorname{Id}$ has
separable range.
Suppose that there are uncountably many points of~$\alpha K_\A$  which are not fixed under~$\phi$. Then
we can choose distinct $B_\xi\in \A$ for $\xi<\omega_1$ such that
$\phi(y_{B_\xi})\ne y_{B_\xi}$. Set $z_\xi=\phi(y_{B_\xi})$.
Since the set
$\{z_\xi: \xi<\omega_1\}$ is either countable or uncountable, by passing to a suitable uncountable subset of it, we may assume that
\begin{enumerate}
\item either $z_\xi=z_\eta=z$ for all $\xi,\eta<\omega_1$ and some $z\in\alpha K_\A$, in which case we may further arrange that $z\in U(B_\xi)$ if and only if $z\in U(B_\eta)$ for any $\xi, \eta<\omega_1$,
\item or  by a transfinite recursive construction, there are distinct $A_\xi\in \A$ for $\xi<\omega_1$ such that
  $z_\xi=y_{A_\xi}$ and $A_\xi\not=B_\eta$ for any  $\xi, \eta<\omega_1$.
\end{enumerate}
Then, for $\xi<\eta<\omega_1$, we have
\begin{align*} \|(T_\phi-\lambda\operatorname{Id})(1_{U(B_\xi)})
-(T_\phi-\lambda\operatorname{Id})(1_{U(B_\eta)})\|
&\geq\\ |(T_\phi-\lambda\operatorname{Id})(1_{U(B_\xi)})(y_{B_\xi})
-(T_\phi-\lambda\operatorname{Id})(1_{U(B_\eta)})(y_{B_\xi})|&=\\
|1_{U(B_\xi)}(z_\xi)-\lambda-1_{U(B_\eta)}(z_\xi))|
&=|\lambda|
\end{align*}
because in case (1) $1_{U(B_\xi)}(z_\xi)=1_{U(B_\eta)}(z_\xi)$, while in case (2) $z_\xi=y_{A_\xi}\notin\{y_{B_\xi}, y_{B_\eta}\}$. Since $\lambda$ was chosen such that
 $T_\phi-\lambda\operatorname{Id}$ has
separable range, we must have $\lambda=0$, so~$T_\phi$ has separable range, as desired.
\end{proof}

\subsection{Automatic continuity of homomorphisms from $\mathscr{B}(C_0(K_\mathcal{A}))$}\label{autocont}
Ever since John\-son~\cite{johnson} proved that every algebra homomorphism from $\BB(\X)$ into a Banach algebra is continuous whenever the Banach space~$\X$ is isomorphic to its square~\mbox{$\X\oplus \X$}, it has been a natural question whether this is also true for infinite-dimensional Banach spaces~$\X$ which are not isomorphic to their squares. The first examples of Banach spaces with that property were not found until 1960, just a few years before Johnson's paper. They were: 
\begin{itemize}
  \item the quasi-reflexive Banach space~$\mathcal{J}$ constructed by James~\cite{james}, as shown by Bessaga and Pe\l{}czy\'{n}ski~\cite{bessaga}, and
\item the space of continuous functions $C([0,\omega_1])$, as shown by 
Semadeni~\cite{semadeni-omega1}, where~$[0,\omega_1]$ denotes the
compact Hausdorff space consisting of all ordinals not exceeding~$\omega_1$, endowed with the order
topology.
\end{itemize}
Willis~\cite{willis} and Ogden~\cite{ogden} have 
proved that all
 homomorphisms from $\BB(\X)$ into a Banach algebra are indeed
 continuous for these two spaces.
 However, this conclusion does not
 extend to all Banach spaces because  Read~\cite{read} has constructed a Banach
 space~$\mathcal{X}_{\text{R}}$ such that~$\mathscr{B}(\mathcal{X}_{\text{R}})$ admits a discontinuous
 derivation, and hence a discontinuous homomorphism into a Banach algebra.  Dales, Loy and
 Willis~\cite{dlw} have subsequently given an example of a Banach
 space~$\mathcal{X}_{\text{DLW}}$ such that all derivations
 from~$\mathscr{B}(\mathcal{X}_{\text{DLW}})$ are continuous, but, assuming the Continuum Hypo\-thesis~\textsf{(CH)},  there is a discontinuous
 homo\-mor\-phism from~$\mathscr{B}(\mathcal{X}_{\text{DLW}})$ into a Banach algebra.

  When $K$ is an infinite metrizable compact space, or in other words when $C(K)$ is an infinite-dimensional separable Banach space, all homomorphisms from $\BB(C(K))$ into a Banach algebra
are
continuous by Johnson's original result because the isomorphic classification of separable $C(K)$-spaces due to
Milutin, Bes\-sa\-ga and Pe\l\-czy\'n\-ski implies that $C(K)$ is isomorphic to
its square.

On the other hand, the Banach space~$C_0(K_\A)$ obtained in~\cite{mrowka} or
Theorem~\ref{main-theorem} is clearly not isomorphic to its square, so the question arises whether every homomorphism from~$\BB(C_0(K_\A))$ into a Banach algebra is continuous. The aim of this subsection is to show that this is indeed true. Our proof uses standard techniques, following closely the approach of Willis~\cite{willis}, which is based on the following notion.

\begin{definition}\label{Defncompress}
  A bounded linear operator $T$ on a Banach space~$\mathcal{X}$ is \emph{compressible} if, for some integer~$n\ge 1$, there is a sequence of bounded linear projections $(Q_j)$ on the direct sum~$\mathcal{X}^n$ of~$n$ copies of~$\mathcal{X}$ such that:
  \begin{enumerate}[label={\normalfont{(\roman*)}}]
  \item\label{Defncompress1} $Q_jQ_k=0$ whenever $j,k\in\N$ are distinct, and
  \item\label{Defncompress2} $T$   factors through $Q_j$ for every $j\in\N$.
  \end{enumerate}
\end{definition}

Before we state our central result, which will have the desired conclusion as an immediate consequence, recall that
$\mathscr{X}(C_0(K_\A))$ denotes the ideal of~$\BB(C_0(K_\A))$ consisting of all operators with separable range.   By Lemma~\ref{jl-c0}, an operator belongs to this ideal if and only if it factors through~$c_0$.

\begin{lemma}\label{AutoContinuityLemma} Let $\A\subseteq\sub$ be an almost disjoint family. Then: 
  \begin{enumerate}
  \item\label{AutoContinuityLemma1} Every operator in the ideal~$\mathscr{X}(C_0(K_\A))$ is compressible.
  \item\label{AutoContinuityLemma2} Null sequences in $\mathscr{X}(C_0(K_\A))$ factor in the following precise sense: whenever $(T_j)$ is a sequence in  $\mathscr{X}(C_0(K_\A))$ such that $\lVert T_j\rVert\to 0$ as $j\to\infty,$ there are $S\in\mathscr{X}(C_0(K_\A))$ and a sequence $(R_j)$ in~$\mathscr{X}(C_0(K_\A))$ such that $T_j = SR_j$ for every $j\in\N$ and $\lVert R_j\rVert\to 0$ as $j\to\infty$.
  \item\label{AutoContinuityLemma3} Every homomorphism from~$\mathscr{X}(C_0(K_\A))$ into a Banach algebra is continuous. 
\end{enumerate}
\end{lemma}

\begin{proof} Since~$C_0(K_\A)$ contains a complemented copy of~$c_0$, we can find bounded linear operators $U\colon C_0(K_\A)\to c_0$ and $V\colon c_0\to C_0(K_\A)$ such that $UV = \Id_{c_0}$.  

  (\ref{AutoContinuityLemma1}). Suppose that $T\in\mathscr{X}(C_0(K_\A))$. We shall show that Definition~\ref{Defncompress} is satisfied for $n=1$. The fact that~$c_0$ is isomorphic to the $c_0$-direct sum of countably many copies of itself means that we can find a sequence $(P_j)$ of projections in~$\BB(c_0)$ such that $P_jP_k=0$ whenever $j,k\in\N$ are distinct and $P_j[c_0]\sim c_0$ for every $j\in\N$. Then  the operators $Q_j = VP_jU\in\BB(C_0(K_\A))$ for $j\in\N$ are projections which satisfy part~\ref{Defncompress1} of Definition~\ref{Defncompress} because  $UV = \Id_{c_0}$, and this fact also implies that $Q_j[C_0(K_\A)]\sim P_j[c_0]\sim c_0$ for every  $j\in\N$.

To verify part~\ref{Defncompress2} of Definition~\ref{Defncompress}, fix $j\in\N$.
  By Lemma~\ref{jl-c0}, we can find a closed subspace~$\mathcal{Y}$ of~$C_0(K_\A)$ such that $T[C_0(K_\A)]\subseteq\mathcal{Y}$ and $\mathcal{Y}\sim c_0$, and  hence there is an isomorphism $W\colon\mathcal{Y}\to Q_j[C_0(K_\A)]$. Define $R = JWT\in\BB(C_0(K_\A))$ and $S = J'W^{-1}Q_j\in\BB(C_0(K_\A))$, where $J\colon Q_j[C_0(K_\A)]\to C_0(K_\A)$ and $J'\colon\mathcal{Y}\to C_0(K_\A)$ are the inclusions. Then it is easy to see that $T = SQ_jR$, as desired. (Alternatively, this can be shown using Proposition~1 in~\cite{willis}.) 
   
  (\ref{AutoContinuityLemma2}). Let $(T_j)$ be a null sequence in  $\mathscr{X}(C_0(K_\A))$. Since $\bigcup_{j\in\N} T_j[C_0(K_\A)]$ spans a separable subspace of~$C_0(K_\A)$, Lemma~\ref{jl-c0} implies that there is a closed subspace~$\mathcal{Y}$ of~$C_0(K_\A)$ such that $\bigcup_{j\in\N} T_j[C_0(K_\A)]\subseteq\mathcal{Y}$ and $\mathcal{Y}\sim c_0$. Let $W\colon\mathcal{Y}\to c_0$ be an isomorphism, and define $R_j = VWT_j\in\mathscr{X}(C_0(K_\A))$ for $j\in\N$ and $S = JW^{-1}U\in\mathscr{X}(C_0(K_\A))$, where $J\colon\mathcal{Y}\to C_0(K_\A)$ is the inclusion and we recall that the operators~$U$ and~$V$ were chosen at the beginning of the proof. Then $SR_j = T_j$ for every $j\in\N$ because $UV = \Id_{c_0}$,  and 
$\lVert R_j\rVert\leq\lVert V\rVert\,\lVert W\rVert\,\lVert T_j\rVert\to 0$ as $j\to\infty$,  as desired.   

  (\ref{AutoContinuityLemma3}). Let $\theta\colon\mathscr{X}(C_0(K_\A))\to\mathscr{C}$ be a homomorphism into a Banach algebra~$\mathscr{C}$, and consider the associated continuity ideal, which is defined by
\begin{align*} \mathscr{I} = \{ S\in\mathscr{X}(C_0(K_\A)) :\  &\text{the maps}\ T\mapsto\theta(ST),\,\mathscr{X}(C_0(K_\A))\to\mathscr{C},\\ &\text{and}\  T\mapsto\theta(TS),\, \mathscr{X}(C_0(K_\A))\to\mathscr{C},\ \text{are continuous}\}. \end{align*}
Suppose that $S\in\mathscr{X}(C_0(K_\A))$.
Applying part~(\ref{AutoContinuityLemma2}) to the trivial null sequence
$(S,0,0,\ldots)$ in $\mathscr{X}(C_0(K_\A))$, we see that~$S$ can be written as
$S = RR'$ for some $R,R'\in\mathscr{X}(C_0(K_\A))$.  The same argument applied to~$R'$ instead of~$S$ shows that $R' = TT'$ for some $T,T'\in\mathscr{X}(C_0(K_\A))$. By Proposition~7  of~\cite{willis}, we have \mbox{$RTT'\in\mathscr{I}$} whenever  $T\in\BB(C_0(K_\A))$ is compressible and $R,T'\in\mathscr{X}(C_0(K_\A))$.
Since every
operator~$T$ in~$\mathscr{X}(C_0(K_\A))$ is compressible
by part~(\ref{AutoContinuityLemma1}), we conclude that \mbox{$S = RTT'\in\mathscr{I}$}.

We can now show that~$\theta$ is continuous. Suppose that $(T_j)$ is a null sequence in
$\mathscr{X}(C_0(K_\A))$. By
part~(\ref{AutoContinuityLemma2}), we can find
$S\in\mathscr{X}(C_0(K_\A))$ and a null sequence~$(R_j)$
in~$\mathscr{X}(C_0(K_\A))$ such that $T_j = SR_j$ for every
$j\in\N$. As shown above, we have $S\in\mathscr{I}$, so the map
$R\mapsto\theta(SR),$ $\mathscr{X}(C_0(K_\A))\to\mathscr{C},$ is
continuous, and therefore
\[ \theta(T_j) = \theta(SR_j)\to\theta(S\circ 0)=0\quad\text{as}\quad j\to\infty. \qedhere \]
\end{proof}

\begin{corollary}\label{CorAutoCty}
 Let $\A\subseteq\sub$ be an almost disjoint family such that the ideal $\mathscr{X}(C_0(K_\A))$ has finite codimension in~$\BB(C_0(K_\A))$. Then every homomorphism from~$\BB(C_0(K_\A))$ into a Banach algebra is continuous. 
\end{corollary}  

\begin{proof}
  Let $\theta$ be a homomorphism from~$\BB(C_0(K_\A))$
  into a Banach algebra.
  Its re\-stric\-tion to~$\mathscr{X}(C_0(K_\A))$  is continuous by part~(\ref{AutoContinuityLemma3}) of  Lemma~\ref{AutoContinuityLemma}, and hence~$\theta$ is continuous because~$\mathscr{X}(C_0(K_\A))$ is closed and has finite codimension in~$\BB(C_0(K_\A))$.  
\end{proof}  

\begin{remark}\label{discontinuous}
  None of the Banach spaces~$\X$ such that $\BB(\X)$ admits a
  discontinuous homomorphism into a Banach algebra that we described
  at the beginning of this subsection have the form~$C(K)$ for a
  compact Hausdorff space~$K$. Hence, one may ask whether every
  homomorphism from~$\BB(C(K))$ into a Banach algebra is continuous
  whenever $K$ is a compact Hausdorff space.  The following example
  shows that under~\textsf{CH} this is not true. We do not know of such an example within \textsf{ZFC}.

  The first named author has constructed a compact Hausdorff space~$K$ such that every
bounded linear operator on~$C(K)$ is the sum of a multi\-pli\-ca\-tion
opera\-tor and a weakly compact operator (see Theorem~6.1 of~\cite{few}). Moreover, $K$ can be chosen
without isolated points, in which case the quotient of~$\BB(C(K))$
by the ideal of weakly compact operators is isomorphic
to~$C(K)$ as a Banach alge\-bra by Theorem~6.5(i) of \cite{dkkkl}. Hence  there is a surjective, continuous homomorphism~$\pi\colon \BB(C(K))\to C(K)$. A~famous result of Dales~\cite{HGDajm} and Esterle~\cite{esterle} states that under~\textsf{CH}, $C(K)$ admits a discontinuous homomorphism~$\theta$ into a Banach algebra  whenever~$K$ is an infinite compact Hausdorff space. Since~$\pi$ is surjective, it follows that the composition $\theta\circ\pi$ is a discontinuous homomorphism from~$\BB(C(K))$ into a Banach algebra.

Note that
the original
construction in~\cite{few} of the space~$K$ also assumed~\textsf{CH}, but Ple\-ba\-nek~\cite{plebanek} has
subsequently modified it to remove that assumption.
A survey of this body of work is given in~\cite{fewsur}.
\end{remark}

\subsection{Banach spaces with few operators and characters on $\BB(\X)$}\label{few}
The question ``Which kinds of bound\-ed linear operators must exist on
a Ba\-nach space?'' has a long history, culminating in the
spectacular resolution of the ``scalar-plus-compact'' problem   a decade
ago by Argyros and Haydon~\cite{argyros-haydon}, who produced a Banach
space on which every bounded linear operator is a compact perturbation
of a scalar multiple of the identity.
A key ingredient, and the
seminal result in this line of research, is the construction by Gowers
and Maurey~\cite{gowers-maurey} of a Banach space on which every
bounded linear operator is a strictly singular perturbation of a
scalar multiple of the identity,  where we recall that a bounded linear
operator is \emph{strictly singular} if no restriction of it to an
infinite-dimensional subspace is an isomorphism onto its range. We refer to~\cite{maurey} for a much more detailed survey of the work of Gowers and Maurey and its context. 

Perhaps the earliest construction of a Banach space with ``few
operators'' is due to Shelah~\cite{sh}, who found a nonseparable
Banach space on which every bounded linear operator is the sum of a
scalar multiple of the identity and an operator with separable
range. Shelah's original example relied on an additional set-theoretic
axiom,~$\diamondsuit$, but this assumption was later removed by
Shelah and Stepr\={a}ns~\cite{ss}.  Wark~\cite{wark1,wark2} has taken
this line of research further by producing a reflexive and, much more
recently, a uniformly convex space with the above property.  Note
that the space $C_0(K_\A)$ from~\cite{mrowka} or
Theorem~\ref{main-theorem} is another instance of a nonseparable Banach space with the property that every
bounded linear operator on it is the sum of a scalar multiple of the
identity and an operator with separable range.

The above-mentioned Banach spaces of Argyros and Haydon, Gowers and
Maurey, Shelah, Stepr\={a}ns and Wark
all have very complex definitions. By contrast, Banach spaces of the
form~$C(K)$ for a compact Hausdorff space~$K$ are among the simplest
Banach spaces that one can define, so a natural question is: To what extent
can a $C(K)$-space have few operators in any similar sense?

To make this question more precise, we observe that a common feature
of all the different variants of Banach spaces~$\mathcal{X}$ with
few operators that we have described above is that the Banach
algebra~$\BB(\mathcal{X})$ has a maximal ideal of codimension one,
and therefore the quotient map induces a \emph{character}
on~$\BB(\mathcal{X})$, that is, a nonzero
linear functional \mbox{$\varphi: \BB(\X)\rightarrow\K$} which is multiplicative
in the sense that $\varphi(TS)=\varphi(T)\varphi(S)$ for all
$S,T\in\BB(\X)$. Clearly, the kernel of a character is an  ideal
of codimension one, thus a maximal ideal, and it is closed.

Characters 
on the Banach algebra~$\BB(\mathcal X)$ are not common  due to the following result (see, e.g., Theorem~2.5.11 of~\cite{HGD}):

\begin{theorem}\label{Nocharsoncartesianspaces} 
Suppose that $\mathcal X$ is a Banach space which is isomorphic to its square $\mathcal{X}\oplus \mathcal{X}$.
  Then $\BB(\mathcal{X})$ has no closed two-sided ideals of finite codimension,
  and hence there are no char\-ac\-ters on~$\BB(\X)$.
\end{theorem}

Recall from Subsection~\ref{autocont} that James' quasi-reflexive
space~$\mathcal{J}$ and $C([0,\omega_1])$ were the first
infinite-dimensional Banach spaces shown not to be isomorphic to their
squares. It turns out that both~$\BB(\mathcal{J})$
and~$\BB(C([0,\omega_1]))$ admit a character. In 1969,
Berkson and
Porta~\cite{bp}  identified 
the character on~$\BB(\mathcal{J})$ explicitly. A year later  Edelstein and Mityagin~\cite{em} described  both  characters. However,  the character
on~$\BB(C([0,\omega_1]))$ can already be found im\-pli\-citly in Se\-ma\-de\-ni's paper,
as explained in Propo\-si\-tion~2.5 of~\cite{omega1}, as well as in 
a subsequent paper of Al\-spach and Ben\-ya\-mini~\cite{alspach}. Loy
and Willis~\cite{lw} have studied these two characters in much greater detail. In both cases, $\BB(\X)$ admits only one character, as shown in~\cite{NJLmaxIdeals} and~\cite{kanialaustsen}, respectively. 

The above-mentioned result of Argyros and Haydon
shows the existence of a Banach space~$\mathcal{X}_{\text{AH}}$ such
that~$\BB(\mathcal{X}_{\text{AH}})$ admits a character with the smallest possible
kernel. (Here we are using the fact that the Banach space they constructed has a Schauder basis, so the ideal of finite-rank
operators, which is contained in every non\-zero two-sided ideal
of~$\BB(\mathcal{X})$, is dense in the ideal of compact operators.) This result, as well as that of
Gowers and Maurey, has no direct analogue for
$C(K)$-spaces, as the following proposition shows.

\begin{proposition}\label{BCK/SCKinfdim}  Suppose that $K$ is an infinite compact Hausdorff space. Then the ideal of strictly singular operators has infinite codimension in~$\BB(C(K))$.
\end{proposition}

\begin{proof} We split in two cases. If the set~$K'$ of  nonisolated points in~$K$ is finite, then \[ C(K)\sim \ell_\infty(K')\oplus c_0(K\setminus K')\sim c_0(K_d), \] where~$K_d$ denotes the set~$K$ endowed with the discrete topology. Hence the conclusion follows from the fact that the  result is true for~$c_0$. 

Otherwise~$K'$ is infinite, and we can recursively construct a sequence $(x_n)$ in~$K'$ and disjoint open sets~$(U_n)$ such that $x_n\in U_n$ for every $n\in\N$. Choose a function $f_n\in C(K)$ such that $f_n(x_n) =1$ and $\operatorname{supp}f_n\subseteq U_n$ for every \mbox{$n\in\N$}. We shall now complete the proof by showing that the multiplication operators \mbox{$M_{f_n}\colon  C(K)\to C(K)$} defined by $M_{f_n}(g)=f_ng$ are linearly independent over the ideal of strictly singular operators. To this end,  suppose that  $S = \sum_{j=1}^n \lambda_j M_{f_j}$ is a strictly singular operator for some $n\in\N$ and some scalars $\lambda_1,\ldots,\lambda_n$. We observe that~$S$ is equal to the multiplication operator~$M_f$, where \mbox{$f=  \sum_{j=1}^n \lambda_j f_j\in C(K)$}. The strict singularity of $S=M_f$ means that the set $\{ x\in K  :\lvert f(x)\rvert\ge\varepsilon\}$ is finite for every $\varepsilon>0$. In particular~$f$ vanishes at every nonisolated point of~$K$, so $0=f(x_j)=\lambda_j$ for each $j\in\{1,\ldots,n\}$, as required. 
\end{proof}

\begin{remark}\label{Remark10042020}
  \begin{enumerate}
\item  Even though we do not need it here, for context it is perhaps worth remarking that
Pe\l\-czy\'n\-ski~\cite{pelczynski} has shown that  for a compact Hausdorff space~$K$, a bounded linear operator from~$C(K)$ into a Banach space is strictly singular if and only if it is weakly compact.
\item As mentioned in Remark~\ref{discontinuous}, the first named author~\cite{few} has constructed a  compact Hausdorff space~$K$  such that every
bounded linear operator on~$C(K)$ is the sum of a multi\-pli\-ca\-tion
operator and a weakly compact operator.  
Obviously,  multi\-pli\-ca\-tion operators
are bounded and linear on any $C(K)$-space, so this space may be viewed as a $C(K)$-space with ``few operators'' in a similar  sense to Gowers and Maurey.
\end{enumerate}
\end{remark}

We observe that $\BB(C(K))$ has no
characters when~$K$ is a
metrizable compact space with at least two points because $\BB(C(K))$ is isomorphic to the algebra of scalar $K\times K$ 
matrices if~$K$ is finite, and $C(K)$ is isomorphic to its square otherwise. On the other hand, the map $\varphi$
given by $\varphi(\lambda \Id+S)=\lambda$ is an example of a
character on~$\BB(C_0(K_\A))$, where $\A\subseteq\sub$, $\lambda\in\K$ and
$S\in\mathscr{X}(C_0(K_\A))$ are as in
Theorem~\ref{main-theorem}.

We shall next show that the kernel of this character is as small as possible; see part~(\ref{remarkSmallKer2}) 
of Remark~\ref{remarkSmallKer} below
for the precise meaning of this statement. It in\-volves two additional pieces of notation and terminology concerning  a Banach space~$\X$: First, we write $\mathscr{G}_{c_0}(\X)$ for the set of bounded linear operators on~$\X$ which factor through~$c_0$. This is a two-sided ideal of~$\BB(\X)$, but not in general closed. It is always contained in the ideal~$\mathscr{X}(\X)$ of operators with separable range, which is closed.  When $\X = C_0(K_\A)$ or $\X = C(\alpha K_\A)$ for an almost disjoint family \mbox{$\A\subseteq\sub$}, we recall from Lemma~\ref{jl-c0} that these two ideals are equal, so in particular~$\mathscr{G}_{c_0}(\X)$ is closed in these cases. Second, 
$\X$ is a \emph{Grothendieck space} if
every weak*-con\-vergent sequence in the dual space $\X^*$ converges
weakly.

\begin{proposition}\label{kernelofcharBCK}
  Suppose that $K$ is an infinite compact Hausdorff space and that~$\varphi$
is a character on~$\BB(C(K))$. Then  
$\mathscr{G}_{c_0}(C(K))\subseteq\ker\varphi$.

If $C(K)$ is a Grothendieck space, then 
$\mathscr{X}(C(K))$, and thus $\mathscr{G}_{c_0}(C(K))$, is contained in the ideal of strictly singular operators, and so $\mathscr{X}(C(K))$ and $\mathscr{G}_{c_0}(C(K))$ have infinite codimension in~$\BB(C(K))$.
\end{proposition}

\begin{proof}
  We split the proof in two parts, beginning with the case
  where~$C(K)$ is a Grothen\-dieck space.
 In this case Diestel~\cite{gdiestel} has shown  that every operator $S\in\mathscr{X}(C(K))$ is weakly compact and
  therefore strictly singular by the above-mentioned result of
  Pe\l\-czy\'n\-ski~\cite{pelczynski}. Hence  $S\in\ker\varphi$ because Proposition~6.6
  of~\cite{NJLmaxIdeals} implies that, for every infinite-dimensional Banach space~$\X$,
  every maximal two-sided ideal
  of~$\BB(\mathcal{X})$ contains the ideal of strictly singular
  operators. Now the main conclusion follows from the fact that $\mathscr{G}_{c_0}(C(K))\subseteq\mathscr{X}(C(K))$, while the final statement is a consequence of Proposition~\ref{BCK/SCKinfdim}.

It remains to prove that $\mathscr{G}_{c_0}(C(K))\subseteq\ker\varphi$   when~$C(K)$ is not a Grothen\-dieck space. 
Schachermayer~\cite{schachermayer} has shown that in this case~$C(K)$
contains a complemented copy of~$c_0$, that is, $\BB(C(K))$ contains a projection~$P$ whose range is isomorphic to~$c_0$. If $\varphi(P)\ne 0$, then we would obtain a character on~$\BB(c_0)$, which is impossible
by Theorem~\ref{Nocharsoncartesianspaces} because \mbox{$c_0\sim c_0\oplus c_0$}. Hence
$P\in\ker\varphi$. Now the conclusion follows  because every
operator $T\in\mathscr{G}_{c_0}(C(K))$ can be written as $T=SPR$ for
some $R,S\in\BB(C(K))$, so
$\varphi(T)=\varphi(S)\varphi(P)\varphi(R)=0$.
  \end{proof}

\begin{remark}\label{remarkSmallKer}
\begin{enumerate}
\item\label{remarkSmallKer2} The fact that we have equality in the inclusion
  $\mathscr{G}_{c_0}(C(K))\subseteq\ker\varphi$  for $K=\alpha K_\A$, where  $\A\subseteq\sub$ is the almost disjoint family
constructed in~\cite{mrowka} or Theorem~\ref{main-theorem}, justifies the claim that the character~$\varphi$ on $\BB(C(\alpha K_\A))$ has the ``smallest possible kernel'' in these cases. 
\item Consider the class of  Banach spaces of the form~$C(K)$, where~$K$ is a compact Haus\-dorff space whose $\omega_1$'st Cantor--Bendixson derivative is empty and $C(K)$ is a Lindel\"{o}f space in the weak topology. The first named author and Zie\-li\'{n}\-ski have  shown
 that the question ``Does this class contain a space   on which every bounded linear operator is the sum of a scalar multiple  of the identity and an operator with  separable range?'' is undecidable (see  Corollary~3.3 and Theorem~4.1 of~\cite{PKPZ}). Note that this result does not apply to any space of the form~$C(\alpha K_\A)$ for an almost disjoint family $\A\subseteq\sub$ because these spaces are known not to be Lindel\"{o}f in their weak topology.
\item With a little extra work, a stronger form of Proposition~\ref{BCK/SCKinfdim} can be proved, namely that the quotient of~$\BB(C(K))$ by the ideal~$\mathscr{S}(C(K))$ of strictly singular operators is nonseparable for every infinite compact Hausdorff space~$K$. We sketch the argument here.

If~$C(K)$ contains a complemented copy of~$c_0$, then the conclusion follows easily from the fact that the quotient of~$\BB(c_0)$ by the ideal of compact  operators (which is equal to the ideal of strictly singular operators in this case) is nonseparable.

Otherwise  the set~$K'$ of
nonisolated points in~$K$ is infinite, and~$C(K)$ is a Grothendieck space by the above-mentioned result
of Schachermayer~\cite{schachermayer}. Then~$C(K')$ is also a
Grothendieck space, for example because~$C(K')$ is a quotient
of~$C(K)$ by taking restrictions of functions, and being a Grothen\-dieck space passes to quotients. Therefore $C(K')$ is nonseparable, so the conclusion follows from the general fact (true whether or not~$C(K)$ is a Grothendieck space)
that $\lVert M_f-S\rVert\ge\lVert f|K'\rVert$ for every $S\in\mathscr{S}(C(K))$  and $f\in C(K)$.
 \end{enumerate}
\end{remark}
\subsection{What is the quotient of 
$\BB(C_0(K_\A))$ by the ideal of operators with separable range?}\label{section-question}
The question ``Which unital Banach algebras are isomorphic to the
quotient $\BB(\X)/\mathscr{K}(\X)$ for some Banach space~$\X$?'', where $\mathscr{K}(\X)$ denotes the ideal of compact operators, has
received considerable attention since the break-through of Argyros
and Haydon~\cite{argyros-haydon}, who  expressed the scalar field in this way.  Notably,
Mo\-ta\-kis, Pugli\-si and Zi\-si\-mo\-pou\-lou~\cite{mpz} have shown that the algebra~$C(K)$ for every countably
infinite compact Hausdorff space~$K$ can also be realized in this form, while \cite{tarbard}
and~\cite{tomek-niels-ind} contain similar conclusions for certain nonsemisimple 
finite-dimensional algebras.

We propose the
following question as a counterpart of the above question
in the present context, where we recall that for an almost disjoint family $\A\subseteq\sub$, Lemma~\ref{jl-c0} implies that the ideal $\mathscr{X}(C_0(K_\A))$  is equal  to the ideal of  operators which factor through~$c_0$.

\begin{question}\label{main-question} Which unital Banach algebras of density at most~$\mathfrak{c}$ can be isomorphic to the quotient algebra 
  $\BB(C_0(K_\A))/\mathscr{X}(C_0(K_\A))$ for some
  uncountable, almost disjoint family $\A\subseteq\sub$?
\end{question}

\noindent The upper bound on the density of the quotient algebra follows from the fact that $\BB(C_0(K_\A))$ has density~$\mathfrak{c}$.

Theorem~\ref{main-theorem} shows that the above quotient may be one-dimensional, that is, isomorphic to the scalar field~$\K$. Further, for every integer $n\ge 2$, we can iterate our construction to realize the algebra of scalar $n\times n$ matrices as such a quotient by identifying~$\N$ with the disjoint union of~$n$ copies of itself and 
take a copy of the family~$\A$ from Theorem~\ref{main-theorem} in each of these~$n$ copies of~$\N$.

We can obtain more 
specific versions of Question~\ref{main-question} by seeking to realize only the Banach algebras belonging to a particular family, or  simply particular Banach algebras, 
as such quotients. Important cases are: 
\begin{enumerate}[label={\normalfont{(\roman*)}}]
  \item The family of unital, separable $C^*$-algebras.
  \item The family of commutative, unital, separable $C^*$-algebras, that is, algebras of the form~$C(K)$ for
    a compact metric space~$K$.
\item The Banach algebra $\BB(\X)$ for
  some known infinite-dimensional Banach space $\X$ such as $\X=\ell_2$, or more generally $\X=\ell_p$ for some $1\le p<\infty$, $\X = c_0$   or $\X=C([0,1])$.
\end{enumerate}

Note that if we could construct an almost disjoint family
$\A\subseteq [\N]^\omega$  such that the quotient   algebra  $\BB(C_0(K_\A))/\mathscr{X}(C_0(K_\A))$ is isomorphic to~$C(K)$ for some infinite compact Hausdorff space~$K$, then under ${\sf CH}$ we could construct a discontinuous
homomorphism from $\BB(C_0(K_\A))$ into a Banach algebra as in Remark~\ref{discontinuous}. This, together with
Corollary~\ref{CorAutoCty},
motivates the following question.
\begin{question}\label{Q47} Suppose that $\A\subseteq\sub$ is an  uncountable, almost disjoint family. Is every homomorphism from~$\BB(C_0(K_\A))$ into a Banach algebra continuous? 
\end{question}

Note that the corresponding question concerning homomorphisms from~$C(K)$ for an infinite compact Hausdorff space~$K$ is known to be  undecidable in \textsf{ZFC}.\smallskip

There is a considerable body of work which may help shed light on
Ques\-tion~\ref{main-question}, as we shall now explain. Throughout, $\A\subseteq\sub$ denotes an uncountable, almost disjoint family. 
We recall that such a family is called \emph{maximal} if it is not properly contained in any other almost disjoint family in~$\sub$. 

Let $C_b(K_\A)$ be the set of all scalar-valued, bounded continuous
functions defined on~$K_\A$. This is a unital, commutative $C^*$-algebra which is
isomorphic to~$C(\beta K_\A)$ for
the \v{C}ech--Stone compactification~$\beta K_\A$ of $K_\A$. Clearly $C_b(K_\A)$
contains~$C_0(K_\A)$ as a closed ideal.  In the language of $C^*$-algebras, the quotient
$C_b(K_\A)/C_0(K_\A)$ is known as the \emph{corona algebra,} while $C_b(K_\A)$ is the \emph{multiplier algebra} of~$C_0(K_\A)$. 

Every function $g\in C_b(K_\A)$ induces a bounded linear operator~$M_g$ on $C_0(K_\A)$ by
mul\-ti\-pli\-ca\-tion, that is, $M_g$ is given by $M_g(f)=fg$. The map $g\mapsto M_g$ is an iso\-metric algebra homomorphism of~$C_b(K_\A)$ into~$\BB(C_0(K_\A))$. We write~$\MM(C_0(K_\A))$ for the range of this map, which is isometrically isomorphic to~$C_b(K_\A)$ and thus to~$C(\beta K_\A)$.

\begin{lemma}\label{mad} Suppose that 
$\A\subseteq [\N]^\omega$ is a maximal almost disjoint family and that
  $g\in C_b(K_\A)$.  Then $g\in C_0(K_\A)$ if and only if the
  multiplication operator~$M_g$ has separable range.
Consequently  the corona algebra $C_b(K_\A)/C_0(K_\A)$ is isomorphic~to
\[ \frac{\MM(C_0(K_\A))}{\MM(C_0(K_\A))\cap \mathscr{X}(C_0(K_\A))}. \]
\end{lemma}

\begin{proof}
Lemma~\ref{separability-lemma} implies that  the operator~$M_g$ has separable range if and only if the set $s(g)=\{A\in \A: 
g(y_A)\not=0\}$ is countable, so we seek to prove that the latter statement is equivalent to $g\in C_0(K_\A)$. One implication is clear, namely that $s(g)$ is countable for every $g\in C_0(K_\A)$.
  
Conversely, suppose that $g\in C_b(K_\A)\setminus C_0(K_\A)$. Then
$L=\{x\in K_\A: |g(x)|\geq\varepsilon\}$ is noncompact for some
$\varepsilon>0$. Define \[ B=\{n\in \N:
|g(x_n)|\geq\varepsilon/2\}\quad\text{and}\quad \A' = \{A\in\A :
A\cap B\in\sub\}. \] First, we observe that $B'= B\setminus\bigcup\A'$ is finite 
because otherwise $\A\cup\{B'\}$ would be an almost disjoint family in~$\sub$ strictly larger than~$\A$, contradicting the maximality of~$\A$. Second, $\A'$ cannot be finite because if it were, $L$ would be
compact. Hence $\B=\{A\cap B: A\in\A'\}$ is an infinite family of
infinite, almost disjoint subsets of~$B$. The maximality of~$\A$
implies that~$\B$ is maximal as well, so~$\B$ must be uncountable because
no countably infinite set contains a countably infinite maximal almost disjoint family. Hence~$\A'$ is also uncountable. The definition of~$B$ implies that 
$g(y_A)\ne0$ for every $A\in \A'$, so $\A'\subseteq s(g)$, proving that~$s(g)$ is uncountable.
\end{proof}

\begin{remark}\label{nonmad}
  The conclusion of Lemma~\ref{mad} fails if the almost disjoint family \mbox{$\A\subseteq [\N]^\omega$} is not maximal. Indeed, in that case  we can choose $B\in[\N]^{\omega}$ such that $A\cap B\in[\N]^{<\omega}$ for every $A\in\A$. Then $F = \{ x_n : n\in B\}$ is a non-compact, clopen subset of~$K_\A$, so $1_F\in C_b(K_\A)\setminus C_0(K_\A)$, but the corresponding multiplication operator~$M_{1_F}$ has separable range.    
\end{remark}  

It is well known that the corona algebra is isometrically isomorphic to
$C(K_\A^*)$, where $K_\A^*=\beta K_\A\setminus K_\A$ is the \v Cech--Stone remainder of $K_\A$.
In the case of a maximal almost disjoint family~$\A\subseteq\sub$, Lemma~\ref{mad} implies that we can define an injective algebra homomorphism
$\psi\colon C_b(K_\A)/C_0(K_\A)\to\mathscr{B}(C_0(K_\A))/\mathscr{X}(C_0(K_\A))$  of norm~$1$ by 
\[\psi(g+C_0(K_\A))= M_g+\mathscr{X}(C_0(K_\A)) \]
for every $C_b(K_\A)$. It can be shown that~$\psi$ is isometric, and hence  $C(K_\A^*)$ is isometrically isomorphic to a commutative subalgebra of  $\BB(C_0(K_\A))/\mathscr{X}(C_0(K_\A))$.

The \v Cech--Stone remainder of~$K_\A$ for a maximal almost disjoint
family~$\A\subseteq\sub$ has been well investigated in the literature.
For example, Kulesza and Levy~\cite{kulesza} have observed that the
methods of~\cite{partitioners} imply that under {\sf CH}, any
separable compact Hausdorff space is homeomorphic to the \v
Cech--Stone remainder of~$K_\A$ for some maximal almost disjoint
family $\A\subseteq [\N]^\omega$.  On the other hand,
Dow~\cite{dow-remainders} has shown that it is consistent that the \v
Cech--Stone remainder of~$K_\A$ has cardinality at most~$\mathfrak c$ for
every maximal almost disjoint family $\A\subseteq\sub$.  In particular
it is undecidable whether~$\beta\N$ is homeomorphic to the \v
Cech--Stone remainder of~$K_\A$ for any maximal almost disjoint family
$\A\subseteq\sub$.

These results suggest that the class of Banach algebras of the form \[
\BB(C_0(K_\A))/\mathscr{X}(C_0(K_\A)) \] may be sensitive to
additional set-theoretic assumptions. To establish that this is indeed the
case would require noncommutative counterparts of the above
consistency results.  The arguments would certainly need to involve
the structure of continuous self-maps of~$K_\A$ and its
compactifications as in Proposition~\ref{maps}, not only scalar-valued
functions on these spaces.  This represents a considerable challenge.
However, the ideas of~\cite{mrowka} to obtain a one-dimensional corona
algebra were successfully adapted to work within \textsf{ZFC} in a stronger, noncommutative
$C^*$-alge\-braic context in~\cite{ext}.

The following question arises  naturally  from the above body of results, including Lemma~\ref{mad}:

\begin{question}  Is there within \textsf{ZFC} a \emph{maximal} almost disjoint family $\A\subseteq[\N]^\omega$
such that $\BB(C_0(K_\A))/\mathscr{X}(C_0(K_\A))$ is one-dimensional?
\end{question}

Mr\'owka's original families of \cite{mrowkas} are maximal, and the
consistently existing families of~\cite{mrowka} can also be
constructed maximal. By contrast, our construction does not give maximality. Indeed, if~$\A\subseteq\sub$ is the family of \cite{mrowkas},
the \v Cech--Stone remainder of~$K_\A$ is a singleton, so the
corona algebra of $C_0(K_\A)$ is one-dimensional, whereas  the
family~$\A$ from Theorem~\ref{main-theorem} fails this property because, by the nonmaximality of~$\A$, some nonzero elements
of the corona algebra correspond to operators with separable
range, as explained in Remark~\ref{nonmad}.\smallskip

A set $B\in \sub$ is called a \emph{partitioner} of an almost disjoint
family~$\A\subseteq \sub$ if, for every $A\in \A$, either 
$A\setminus B$ or
$A\cap B$ 
is finite. A partitioner~$B$  partitions the family~$\A$
into two parts, namely $\A_B=\{ A\in
\A : A\setminus B\in\fin\}$ and its complement  \mbox{$\{ A\in \A : A\cap B\in\fin\}$}. This in turn partitions the space~$K_\A$ into two parts: 
\[ F = \{x_n : n\in B\}\cup\{ y_A :
A\in\A_B\}\quad \text{and}\quad \{x_n : n\in\N\setminus B\}\cup\{ y_A :
A\in\A\setminus\A_B\}, \]
both of which are closed. Hence
the indicator function $1_{F}$ is continuous, so it induces a multiplication operator
\mbox{$M_{1_{F}}\in\BB(C_0(K_\A))$}, which is clearly a projection.  Therefore it decomposes~$C_0(K_\A)$ into the direct sum of two closed subspaces. Lemma~\ref{separability-lemma} shows that
these two subspaces are  nonseparable if and only if both~$\A_B$ and its complement are uncountable.

Partitioners are well studied in the literture (see,
e.g.,~\cite{partitioners}), and almost disjoint families with
``few partitioners'' were already  constructed in the 1940s by
Luzin~\cite{luzin}, who found an almost disjoint family
$\A\subseteq\sub$ which does not admit any partitioner~$B$ such that
both~$\A_B$ and its complement are uncountable.  Note that our almost
disjoint family~$\A$ from Theorem~\ref{main-theorem} shares this
property because~$C_0(K_\A)$ cannot be decomposed into the direct sum
of two nonseparable, closed subspaces. A well-known property of
Mr\'owka's family~$\A$ is that it does not admit any
partitioner~$B$ such that both~$\A_B$ and its complement are
infinite.  In fact, Lemma~\ref{mad} implies that no maximal almost disjoint family~$\A\subseteq\sub$
admits a partitioner~$B$ such that either~$\A_B$ or its complement is
countably infinite.\smallskip

Other results which may help answer
Question~\ref{main-question} address the isomorphic classification
and the structure of complemented subspaces of~$C_0(K_\A)$ for an uncountable, almost disjoint
family $\A\subseteq\sub$. Marciszewski and Pol \cite{mpol} have shown that there
are $2^{\mathfrak c}$ nonisomorphic Banach spaces of this form. Further,
assuming Martin's axiom and the negation of {\sf CH}, we have 
the following two results concerning an almost disjoint family $\A\subseteq\sub$  of uncountable cardinality strictly smaller than~$\mathfrak c$:
\begin{itemize}
\item The Banach space $C_0(K_\A)$ has ``many'' decompositions as a direct sum of
two nonseparable, closed subspaces (see~\cite{mrowka}).
\item   The Banach spaces $C_0(K_\A)$
  and $C_0(K_\B)$ are isomorphic whenever $\B\subseteq\sub$ is an almost disjoint family of the same
  cardinality as~$\A$. In particular, the Banach space~$C_0(K_\A)$ is
  isomorphic to its square (see the recent paper~\cite{sailing} of Cabello S\'anchez, Castillo, Marciszewski, Plebanek and
  Salguero-Alar\-c\'on).
\end{itemize}

\bibliographystyle{amsplain}

\end{document}